\documentclass{article}
\usepackage[top=1in, bottom=1.25in, left=1.25in, right=1.25in]{geometry}

\usepackage{graphicx}
\usepackage{amsmath} 
\usepackage{amssymb}  
\usepackage{graphics} 
\usepackage{graphicx}
\usepackage{epsfig}
\usepackage{amsthm}
\usepackage[utf8]{inputenc}
\usepackage[english]{babel}

\usepackage{hyperref,xcolor}
\usepackage{algorithm}
\usepackage{algorithmic}
\usepackage{cite}

\newtheorem{theo}{Theorem}[section]  
\newtheorem{theorem}{Theorem}

\newtheorem{example}[theo]{Example}

\newtheorem{lemma}[theo]{Lemma}
\newtheorem{remark}[theo]{Remark}
\newtheorem{corollary}[theo]{Corollary}

\newtheorem{assumption}[theo]{Assumption}

\newcommand{\beq}{\begin{equation}}
\newcommand{\eeq}{\end{equation}}
\newcommand{\beqa}{\begin{eqnarray}}
\newcommand{\eeqa}{\end{eqnarray}}
\newcommand{\beqs}{\begin{equation*}}
\newcommand{\eeqs}{\end{equation*}}
\newcommand{\beqas}{\begin{eqnarray*}}
\newcommand{\eeqas}{\end{eqnarray*}}

\newcommand{\R}{\mathbb R}
\newcommand{\N}{\mathbb N}

\newcommand{\bigO}{{\cal O}}
\renewcommand{\P}{\mathbb P}
\newcommand{\E}{\mathbb E}

\newcommand{\hlip}{U}
\newcommand{\dist}{\mbox{dist}}

\newcommand{\Gs}{{G_*}}
\newcommand{\Hs}{{H_*}}
\newcommand{\s}{\sigma}
\newcommand{\sk}{\sigma_k}
\newcommand{\vb}{v}
\newcommand{\mub}{\bar{v}}

\newcommand{\Yqk}{{Y_{q,k}}}
\newcommand{\Yonek}{{Y_{1,k}}}
\newcommand{\Iqk}{{I_{q,k}}}
\newcommand{\aqk}{{\bar{\alpha}_{q,k}}}
\newcommand{\aonek}{{\bar{\alpha}_{1,k}}}
\newcommand{\xqk}{{\bar{x}_{q,k}}}
\newcommand{\xonek}{{\bar{x}_{1,k}}}

\newcommand{\bonek}{b_{1,k}}
\newcommand{\ski}{{\sigma_k(i)}}
\newcommand{\skl}{{\sigma_k(\ell)}}
\newcommand{\aqs}{{a_{q}(s)}}

\newcommand{\xs}{{x^*}}

\newcommand{\bma}{\begin{math}}
\newcommand{\ema}{\end{math}}

\def\rev#1{{{\color{black}#1}}} 

\begin{document} 
\title{Why Random Reshuffling Beats \\Stochastic Gradient Descent
}


\author{M.~G\"urb\"uzbalaban         \and
        A.~Ozdaglar \and P.A.~Parrilo.
}


\author{Mert G\"urb\"uzbalaban\thanks{Department of Management Science and Information Systems, Rutgers University, 100 Rockafellar Road, Piscataway, NJ 08854. E-mail: \texttt{mg1366@rutgers.edu}. The author's work is partially supported by the NSF DMS-1723085 Award.}, 
Asuman Ozdaglar\thanks{Department of Electrical Engineering and Computer Science, Massachusetts Institute of Technology, 77 Massachusetts Avenue, Cambridge, Massachusetts 02139. E-mail: \texttt{asuman@mit.edu}.}, \\
Pablo Parrilo \thanks{Department of Electrical Engineering and Computer Science, Massachusetts Institute of Technology, 77 Massachusetts Avenue, Cambridge, Massachusetts 02139. E-mail: \texttt{parrilo@mit.edu}.}
}



\maketitle

\begin{abstract} We analyze the convergence rate of the random reshuffling (RR) method, which is a randomized first-order incremental algorithm for minimizing a finite sum of convex component functions. RR proceeds in cycles, picking a uniformly random order (permutation) and processing the component functions one at a time according to this order, i.e., at each cycle, each component function is sampled without replacement from the collection. Though RR has been numerically observed to outperform its with-replacement counterpart stochastic gradient descent (SGD), characterization of its convergence rate has been a long standing open question. In this paper, we answer this question by providing various convergence rate results for RR and variants when the sum function is strongly convex. We first focus on quadratic component functions and show that the expected distance of the iterates generated by RR with stepsize $\alpha_k=\Theta(1/k^s)$ for $s\in (0,1]$ converges to zero at rate $\bigO(1/k^s)$ (with $s=1$ requiring adjusting the stepsize to the strong convexity constant). Our main result shows that when the component functions are quadratics or smooth (with a Lipschitz assumption on the Hessian matrices), RR with iterate averaging and a diminishing stepsize $\alpha_k=\Theta(1/k^s)$ for $s\in (1/2,1)$ converges at rate $\Theta(1/k^{2s})$ with probability one in the suboptimality of the objective value, thus improving upon the $\Omega(1/k)$ rate of SGD. Our analysis draws on the  theory of Polyak-Ruppert averaging and relies on decoupling the dependent cycle gradient error into an independent term over cycles and another term dominated by $\alpha_k^2$. This allows us to apply law of large numbers to an appropriately weighted version of the cycle gradient errors, where the weights depend on the stepsize. We also provide high probability convergence rate estimates that shows decay rate of different terms and allows us to propose a modification of RR with convergence rate $\bigO(\frac{1}{k^2})$.  \end{abstract}

\section{Introduction: First-order incremental methods}

We consider the following unconstrained optimization problem where the objective function is the sum of a large number of component functions: 
\beq\label{pbm-multi-agent}
\min f(x) := \sum_{i=1}^m f_i(x) \quad \mbox{s.t.} \quad x \in \R^n
\eeq
with $f_i:\R^n \to \R$. This problem arises in many contexts and applications including regression or more generally parameter estimation problems (where $f_i(x)$ is the loss function representing the error between the output and the prediction of a parametric model)~\cite{Boyd2011AdmmBook,bertsekas2011incremental,bertsekas1997hybrid, Bertsekas1996incremental}, minimization of an expected value of a function (where the expectation is taken over a finite probability distribution or approximated by an $m$-sample average)~\cite{BottouLecun2005, Leroux2012sgd}, machine learning~\cite{Jordan13DistLearningApi, Dickenstein2014, Leroux2012sgd}, or distributed optimization over networks~\cite{RamNedicVeer2007, NedicOzdaglar09,Nedic2007rate}.

One widely studied approach for solving problem~\eqref{pbm-multi-agent} is the {\it deterministic incremental gradient (IG) method}~\cite{Bertsekas99nonlinear,Bertsekas15Book,bertsekas2011incremental}. IG method is similar to the standard gradient method with the key difference that at each iteration, the decision vector is updated incrementally by taking sequential steps  along the gradient of the component functions $f_i$ in a cyclic order. Hence, we can view each outer iteration $k$ as a cycle of $m$ inner iterations: starting from initial point $x_0^0 \in \R^n$, for each $k\geq 0$, we update the iterate $x_i^k$ as 
 \begin{equation}\label{eq-inner-update} x_{i}^k :={x_{i-1}^k - \alpha_k \nabla f_{i} (x_{i-1}^k)} , \qquad i=1,2,\dots,m,
   \end{equation}
where $\alpha_k>0$ is a stepsize with the convention that $x_0^{k+1} = x_{m}^k$. 


Intuitively, it is clear that slow progress can be obtained if the functions that are processed consecutively have gradients close to zero. Indeed, the performance of IG is known to be pretty sensitive to the order functions are processed~\cite[Example 2.1.3]{Bertsekas15Book}.  If there is a favorable order $\sigma$ (defined as a permutation of $\{1,2,\dots,m\}$) that can be obtained by exploiting problem-specific knowledge, the method can be updated to process the functions with this order instead with the iterations:
	 \begin{equation}\label{eq-inner-update-perm} x_{i}^k :={x_{i-1}^k - \alpha_k \nabla f_{\sigma(i)} (x_{i-1}^k)} , \qquad i=1,2,\dots,m.
   \end{equation}
However, in general a favorable order is not known in advance, and a common approach is choosing the indices of functions to process as independent and uniformly distributed samples from the set $\{1,2,\dots,m\}$. This way no particular order is favored, making the method less vulnerable to particularly bad orders. This approach amounts to at each iteration sampling the function indices \textit{with replacement} from the set $\{1,2,\dots,m\}$ and is called the \textit{Stochastic Gradient Descent} (SGD) method, a.k.a. \textit{Robbins-Monro} algorithm~\cite{Robbins:1951ua}. SGD is strongly related to the classical field of stochastic approximation~\cite{kushner2003stochastic}. Recently it has received a lot of attention due to its applicability to large-scale problems and became popular especially in machine learning applications (see e.g.~\cite{Bottou:2012hg, Borges2010BookChapter,BottouLecun2005,Zhang2004}).

An alternative popular approach that works well in practice is following a mixed approach between SGD and IG, sampling the functions randomly but not allowing repetitions, that is sampling the component functions at each iteration \textit{without-replacement}, or equivalently picking a random order at each cycle. Specifically, at each cycle $k$, we draw a permutation $\sigma_k $ of $\{1,2,\dots,m\}$ independently and uniformly at random over the set of all permutations \beq \Gamma = \big\{ \sigma ~:~ \sigma \mbox{ is a permutation of } \{1,2,\dots,m\} \big\} \label{def-Gamma-perm-set}\eeq 
and process the functions with this order: 
\begin{equation}\label{inner-update-wo} x_{i}^k :={x_{i-1}^k - \alpha_k \nabla f_{\sigma_k(i)} (x_{i-1}^k)} , \qquad i=1,2,\dots,m,
   \end{equation} 
 where $\alpha_k>0$ is a stepsize. We set $x_0^{k+1} = x_{m}^k$ as before and refer to $\{x_0^k\}$ as the {\it outer iterates}. This method is called the \textit{Random Reshuffling} (RR) method~\cite[Section 2.1]{Bertsekas15Book} and will be the focus of this paper.


\section{Motivation and summary of contributions} 
Without-replacement sampling schemes are often easier to implement efficiently compared to with-replacement sampling schemes, guarantee that every point in the data set is touched at least once, and often have better practical performance than their with-replacement counterparts~\cite{Bottou:2012hg, bottou2009curiously,recht2013parallel,
feng2012towards, recht12jmlr,BottouSgdTricks}. For instance, Bottou~\cite{bottou2009curiously} empirically compares SGD and RR methods and finds that RR converges with a rate close to $\sim 1/k^2$ whereas SGD is much slower achieving its min-max lower bound of $\Omega(1/k)$ for strongly convex objective functions~\cite{yudin1983problem,Agarwal12}. Many other papers listed above report a similar empirical behavior. This discrepancy in rate between RR and SGD is not only observed for large $m$ but also for small $m$ (as we illustrate in Example~\ref{exam-one}), and understanding it theoretically has been a long-standing open problem~\cite{recht2013parallel,Bertsekas99nonlinear}. 
 
To our knowledge, the only existing theoretical analysis for RR is given by a recent paper of Recht and R\'e~\cite{recht12jmlr} which focuses on least mean squared optimization and formulates a conjecture that would prove that the expected convergence rate of RR is faster than that of SGD. Given $N$ arbitrary positive-definite matrices of dimension $n\times n$, the conjecture says that products of any $K$ matrices chosen from this set of $N$ matrices satisfy a non-commutative arithmetic-geometric mean inequality for every positive integer $N$ and every $K\leq N$. This conjecture has been proven only in some special cases (for $N=2$~\cite{recht12jmlr}, for $N=3$~\cite{Ward16AMGM} and when $N$ is a multiple of 3 and $K=3$~\cite{Zhang14AMGM}). Recht and R\'e also analyze a special case of \eqref{pbm-multi-agent} (that arises when $f_i(x) = (a_i^T x - y_i)^2$ is a quadratic function where $a_i$ is a column vector that is randomly generated according to a random model and $y_i$ is a scalar) and show that after a fixed amount of iterations, the upper bounds on the expected mean square error using without-replacement sampling is smaller than that of with-replacement sampling with high probability on most models of $a_i$ (probabilities are taken with respect to the random data generation model). Despite these advances, there has been a lack of convergence theory for RR that characterizes its convergence rate and explains its fast performance. Analyzing algorithms based on without-replacement sampling such as RR is more difficult than with-replacement based approaches such as SGD. The reason is that the underlying independence assumption for the with-replacement sampling allows a tractable analysis with classical martingale convergence theory~\cite{Moulines:2011vy, PolyakJuditsky92}, whereas without-replacement sampling introduces correlations and dependencies among the sampled gradients and iterates that are harder to analyze~\cite{recht12jmlr}. The aim of our paper is to {\it fill this theoretical gap} for the case when the objective function $f$ in \eqref{pbm-multi-agent} is strongly convex and {\it develop a novel algorithm that can accelerate the convergence further}. We next summarize our contributions. 

We first consider the case when the component functions are quadratics. Building on the recent convergence rate results for the cyclic IG in~\cite{Gur2015IncGrad}, we first present a key result (Theorem~\ref{theo-rate-quadratics}) that provides an upper bound for the distance from the optimal solution of the iterates generated by an incremental method that processes component functions with an {\it arbitrary fixed order} and uses a stepsize $\Theta(1/k^s)$ for $s\in (0,1]$. This upper bound decays at rate $\bigO(1/k^s)$ and depends on the strong convexity constant of the sum function and an order dependent parameter given by a weighted average of Hessian matrices where the weights are given by the sum of the component gradients processed up to that point according to the given order. We use this result to show that the distance to the optimal solution of the iterates generated by RR algorithm with stepsize $\Theta(1/k^s)$, for all $s\in (0,1]$, converges to 0 at rate $\bigO(1/k^s)$ in expectation (where the expectation is over the random sequence of iterates). However, we show that achieving the rate $\bigO(1/k)$ involves adapting the stepsize to the strong convexity constant of the sum function.


We then consider the $q$-suffix averages of the iterates generated by RR for some $q\in (0,1]$ (which is obtained by averaging the last $qk$ iterates at iteration $k$) and show that with a stepsize  $\alpha_k = R/(k+1)^s$ for $s\in (1/2,1)$ and $R>0$, they converge {\it almost surely at rate $\bigO(1/k^s)$ to the optimal solution}. We provide an explicit characterization of the asymptotic rate constant in terms of the averaging parameter $q$, the stepsize parameters $R$ and $s$ and the Hessian matrices and the gradients of the component functions at the optimal solution (parts $(i)$ and $(ii)$ of Theorem~\ref{thm-averaged-rate-quad}). Using strong convexity, this implies an almost sure convergence rate $\Theta(1/k^{2s})$ in the suboptimality of the objective value. Our analysis views RR as a gradient descent method with random gradient errors. The analysis of RR is complicated by the fact that the cumulative gradient error over cycles are dependent. A key step in our proof is to decouple the cycle gradient error into a $\bigO(\alpha_k)$ term independent over cycles and another term that scales as $\bigO(\alpha_k^2)$. This allows us to use strong law of large numbers for a properly weighted average of the cycle error gradient sequence (where the weights depend on the stepsize) and show almost sure convergence of the $q$-suffix averaged iterates. Another key component of our analysis is to adapt the Polyak-Ruppert averaging techniques developed for SGD~\cite{PolyakJuditsky92,Moulines:2011vy} to RR.  

We also provide a high probability convergence rate estimate for the distance of $q$-suffix averages to the optimal solution that consists of two terms, with the first term corresponding to a $1/k^s$ decay of a ``bias" term (where bias is defined as the expected value of the cycle gradient errors of RR which may be non-zero) and the second term representing a $1/k$ decay for $0<q<1$ (and $\log k/k$ decay for $q=1$); see part $(iii)$ of Theorem~\ref{thm-averaged-rate-quad} . These results are obtained by martingale concentration techniques. We use the characterization of the bias to estimate it with a term that can be computed during the RR iterations. We show that subtracting the estimated bias from the averaged RR iterates accelerates the convergence rate further, leaving only the second error term of $1/k$ decay in the iterates (part $(iv)$ of Theorem~\ref{thm-averaged-rate-quad}). Based on this result, we propose a new algorithm which we call the \textit{De-biased Random Reshuffling} (DRR) method that can accelerate the asymptotic convergence rate of RR in the suboptimality of the function values from $\bigO(1/k^{2s})$ to $\bigO(1/k^2)$. 

Finally, in Theorem~\ref{thm-averaged-random-rate} we show that our results in Theorem \ref{thm-averaged-rate-quad} extend to the more general case when component functions are smooth (twice continuously differentiable) under a Lipschitz assumption on the Hessian, which allows us to control the second order term in a Taylor expansion of the gradient. 

\smallskip
\noindent{\textbf{Outline:}} The outline of the paper is as follows. In Section \ref{sec-prelim}, we introduce our approach for analyzing RR, present Polyak-Ruppert averaging and give a motivating example. Section~\ref{sec:quad-case} focuses on the case when component functions are quadratics. We first present a convergence rate estimate for IG with a fixed arbitrary order. We then focus on RR and study convergence of averaged iterates to the optimal solution. Section~\ref{sec:general-case} extends our results to smooth functions. Section~\ref{sec:accelerated-RR} proposes the DRR algorithm that can accelerate RR further. Finally, we conclude with a summary of our work in Section~\ref{sec:conclusion}. Some of the technical lemmas required in the details of the proofs are deferred to Sections~\ref{sec-appendix-exp-rate}, \ref{sec:appendix-quad-case} and~\ref{sec:appendix-general-case} of the Appendix.

\noindent \textbf{Notation}: We study the point-wise dominance of stochastic sequences by deterministic sequences and use the following notation. Let $x_k = x_k(\omega)$ be a stochastic real-valued sequence (where $\omega$ can be thought as the source of randomness) and $y_k$ be a real-valued deterministic sequence. We write $ x_k = \bigO(y_k)  \iff \exists h>0, \exists k_0 \quad \mbox{such that} \quad  |x_k| \leq  h |y_k| \quad \forall k\geq k_0, \forall \omega,$ where $h$ and $k_0$ are independent of $\omega$ (Note that the requirement is that this inequality holds for all $\omega$, not just for almost all $\omega$). When $x_k$ is non-negative for every $\omega$, given another deterministic positive sequence $z_k$, we also introduce the inequality version of this definition: 
	$x_k \leq y_k + o(z_k) \iff \forall \varepsilon>0, \exists k_0(\varepsilon) \quad \mbox{such that}\	\quad z_k^{-1} |x_k(\omega) - y_k| \leq \varepsilon, \quad \forall k\geq k_0(\varepsilon), \forall \omega$
where $k_0$ depends on $\varepsilon$ but is independent of $\omega$.	When $x_k$ is deterministic, these definitions reduce to the standard definitions of $\bigO(\cdot)$ and $o(\cdot)$ for deterministic sequences. For random $x_k$, the only difference is that we require the constants to be independent of the choice of $\omega$. For example, if $x_k$ is uniformly distributed over $[0,10]$, we write $x_k = \bigO(1)$. Throughout the paper, $\| \cdot \|$ denotes the vector or matrix 2-norm (maximum singular value). 

\section{Preliminaries}\label{sec-prelim}
We consider solving problem (\ref{pbm-multi-agent}) with RR method with iterations given in (\ref{inner-update-wo}). Throughout we assume the following:  
\begin{assumption}\label{assump-sum-is-str-cvx} The sum function $f(x)=\sum_{i=1}^m f_i(x)$ is strongly convex, i.e., there exists a constant $c>0$ such that the function $f(x) - \frac{c}{2} \| x\| ^2$ is convex on $\R^n$.\footnote{Such functions arise naturally in support vector machines and other regularized learning algorithms or regression problems (see e.g. \cite{Leroux2012sgd, Rakhlin:2011, Hogwild11})}.
\end{assumption}

Note that this assumption is on the sum function $f$, it does not require the convexity of the individual component functions $f_i$. A consequence of this assumption is that there exists a unique optimal solution to \eqref{pbm-multi-agent} which we denote by $x^*$. Another consequence is that the Hessian at the optimal solution is invertible since
	 \beq \Hs := \nabla^2 f(x^*) \succeq c I_n \succ 0
		\label{def-Hs}	 
	 \eeq 
where $I_n$	is the $n\times n$ identity matrix.

To analyze RR, we view it as a gradient method with random gradient errors and rewrite 
the outer iterations \eqref{inner-update-wo} as 
   \beq \frac{x_0^k - x_{0}^{k+1}}{\alpha_k} = \nabla f(x_0^k) + E_k \label{averaged-iter-randomized}
   \eeq
where 
   \beq E_k := \sum_{i=1}^{m} \bigg( \nabla f_{\ski}(x_{i-1}^k) -  \nabla f_{\ski} (x_0^k) \bigg) \label{def-Ek}
   \eeq
is the cumulative gradient errors associated with the cycle $k$. 
This approach is similar to the analysis of SGD, where one writes each (inner) iteration as a gradient method with error. The key difference that simplifies the analysis of SGD is the fact that the iteration gradient errors at the current iterate are independent (because of independent identically distributed sampling of component function indices) allowing use of martingale central limit theorems to obtain convergence and rate results (see e.g. \cite{kushner2003stochastic, Fabian:1968cc, Chung:1954iy,PolyakJuditsky92}). In contrast, for RR, not only  are the iteration gradient errors dependent (because of sampling a random order at cycle $k$ coupling indices $\sigma_k(i)$ and $\sigma_k(j)$ for $i\ne j$), but also the cycle gradient errors $E_{k_1}$ and $E_{k_2}$ for cycles $k_1\ne k_2$ are dependent as they both depend on the history of the iterates. This necessitates a different line of attack for the convergence analysis of RR.\footnote{There is some literature that analyzes SGD under correlated noise \cite[Ch.~6]{kushner2003stochastic}, but the noise needs to have a special structure (such as a mixing property) which does not seem to be applicable to the analysis of RR.}

A key idea in our analysis is to use a recent upper bound for the convergence rate of cyclic incremental gradient method (see  \cite{Gur2015IncGrad}), which can be generalized to hold for any fixed deterministic order. This bound implies an almost sure upper bound (in fact one that holds for all sample paths) on the distance of the outer iterates $x_0^k$ generated by RR from the optimal solution $x^*$ denoted by
	$\dist_k  =\| x_0^k - x^*\|$
(see Section \ref{quad-conv}). Crucially, this result implies an upper bound in expected distance which is asymptotically $m$ times smaller than the almost sure guarantees on the distance of the iterates.

In analyzing RR, we will also consider the \textit{average} of the outer iterate sequence given by\footnote{It is well known that computing this (moving) average can be done efficiently in a dynamic manner by storing a vector of length $n$.} 
	$ \bar{x}_{k} := \frac{\sum_{j=0}^{k-1} x_0^j}{k}. \label{simple-average}
	$
 We also consider averaging only the most recent iterates, i.e. at iteration $k$, averaging the last $qk$ iterates
for some constant $q\in (0,1]$:
	$$\bar{x}_{q,k} := \frac{\sum_{j=(1-q)k }^{k-1} x_0^j}{qk}, \quad 0 < q\leq 1. $$ 
The generated sequence is referred to as the \textit{$q$-suffix average} of the sequence $x_0^k$. For SGD, it was shown that $q$-suffix averaging with $0<q<1$ leads to better performance then averaging (which corresponds to the $q=1$ case by definition), improving the convergence rate in the suboptimality of the function value from $\log k/k$ to $1/k$~\cite{Rakhlin:2011, Shamir:2012tj}. This is in line with our results in Section~\ref{sec:quad-case} which show faster rate for the $0<q<1$ case. The parameter $q$ can be thought as a measure of how much memory one uses during the averaging process. We define the \textit{$q$-suffix average} of the stepsize in a similar way:
 \beq \bar{\alpha}_{q,k} =\frac{\sum_{j=(1-q)k}^{k-1} \alpha_j}
{qk}, \quad 0<q\leq 1.\label{def-averaged-stepsize}
\eeq

We will obtain our strongest convergence results (in the almost sure sense and with a similar $m$ dependence as the expected guarantees) for averaged iterate sequences with ``large step sizes", a technique known as Polyak-Ruppert averaging, which has been used in achieving optimal rates for SGD in a robust manner as explained next.

\subsection{Polyak-Ruppert averaging} SGD has a long history going back to the seminal paper of Robbins and Monro~\cite{Robbins:1951ua}. It has been analyzed under different assumptions extensively in the stochastic approximation literature (see e.g.~\cite{kushner2003stochastic}). 
For stochastic convex optimization, it has been shown that SGD has a min-max lower bound of $\Omega(1/k)$~\cite{yudin1983problem,Agarwal12}. One way of achieving this optimal $1/k$ rate is to use a stepsize $\alpha_k = R/k$ where $R$ is a positive scalar adjusted properly to the strong convexity constant of the objective function~\cite{Chung:1954iy,Fabian:1968cc,kushner2003stochastic} but this requires the knowledge or the estimate of an accurate lower bound on the strong convexity constant. If a lower bound is not known or cannot be estimated accurately, the convergence can be potentially slow~\cite[Section 2.1]{Nemirovski:2009kb}. Polyak-Ruppert averaging is a technique that allows to get the optimal $\sim 1/k$ rate in an asymptotically efficient manner without the need to adjust to the strong convexity constant. It relies on using a larger stepsize $\alpha_k = R/k^s$ (with $R$ an arbitrary positive constant and $s\in (1/2,1)$) that decays slower than $\Theta(1/k)$ but then taking the time average of the iterates to filter out the undesired oscillations arising due to the larger steps ~\cite{PolyakJuditsky92, Nemirovski:2009kb, kushner2003stochastic}.\footnote{IG shows similar properties to SGD in terms of the robustness of the stepsize rules $\alpha_k = R/k^s$. The convergence rate (in $k$) is only robust to the strong convexity constant of the objective for $s<1$ but not for $s=1$~\cite{Gur2015IncGrad}.} We will later show that the same technique allows us to get
almost sure guarantees for the averaged iterates without the need to tune the stepsize to the strong convexity constant (see Theorems~\ref{thm-averaged-rate-quad} and~\ref{thm-averaged-random-rate}).

\subsection{A motivating example} 
Before presenting our convergence analysis, we consider a simple example that highlights the difference in convergence mechanisms of SGD and RR and gives intuition on why RR is faster than SGD asymptotically.
\begin{example}\label{exam-one} Consider the component functions
	 \beq f_1(x) = \frac{1}{2} (x-1)^2, \quad f_2(x) = \frac{1}{2}(x+1)^2 + \frac{x^2}{2}
	 \eeq
with $f(x) = f_1(x) + f_2(x) = \frac{3}{2}x^2+1$ and $x^* = 0$. 
The outer RR iterates $\{x_0^k\}$ satisfy
       \beqa   \quad x_0^{k+1} &=& x_0^k -  {\alpha_k} \left(\nabla f_{{\sigma}_k(1)}(x_0^k) +f_{{\sigma}_k(2)}(x_1^k)\right) = x_0^k - {\alpha_k} (\nabla f(x_0^k) + E_k), 
       \label{eq-simple-component-func}
       \eeqa
where the cycle gradient errors are given by
\beq E_k = \begin{cases} 
	\nabla f_2(x_1^k) - \nabla f_2(x_0^k)  &  \quad \mbox{with probability } 1/2, \quad \mbox{for } \sigma_k = \{1,2\}, \\ 
	 \nabla f_1(x_1^k) - \nabla f_1(x_0^k)   &  \quad \mbox{with probability } 1/2, \quad \mbox{for } \sigma_k = \{2,1\}.
\end{cases} \label{def-Ek-example}
 \eeq 
 Plugging in the identities $\nabla f_1(x) = x - 1$, $\nabla f_2(x) = 2x + 1$ obtained from \eqref{eq-simple-component-func} and the inner update formula \eqref{inner-update-wo}, we obtain  
    \beq E_k = \alpha_k \mu(\sigma_k) - 2\alpha_k x_0^k \label{grad-error-exam}
    \eeq 
where $\mu(\sigma_k) = - \nabla^2 f_{\sigma_k(2)} (x^*) \nabla f_{\sigma_k(1)}(x^*)$ satisfying 
$$\mu(\sigma_k) = \begin{cases} 
	+2  &  \quad \mbox{with probability } 1/2, \quad \mbox{for } \sigma_k = \{1,2\}, \\ 
	 -1  &  \quad \mbox{with probability } 1/2, \quad \mbox{for } \sigma_k = \{2,1\}.
\end{cases}  
 $$ 
In contrast, SGD starting from an initial point $y^0$ leads to the iterations
\beq   \quad y^{j+1} = y^j - \alpha_j \nabla f_{i_j}(y^j) =  y^j - \frac{\alpha_j}{2} (\nabla f(y^j) + e^j), 
       		\label{sgd-iters-2} 
 \eeq
where $ i_j$ is an independent and identically distributed (i.i.d.) random variable with a uniform distribution over the index set $\{1,2\}$ and the gradient error $e^j$ is given by
\beq{e}^j = \begin{cases} 
	-2 - y^j  &  \quad \mbox{with probability } 1/2, \quad \mbox{for } i_j = 1, \\
	~~2 + y^j &  \quad \mbox{with probability } 1/2, \quad \mbox{for } i_j = 2.
\end{cases}\label{eq-sgd-exam-error}
 \eeq

We consider a stepsize of $\alpha_k = {R\over k^s}$ with $s=0.75$ for both algorithms. Note that for this example, RR is globally convergent to the optimal solution $x^*=0$ with probability one, therefore $x_0^j \to 0$.\footnote{To see this, note that the RR iterations for this example are given by
$ x_0^{k+1} = (1-\frac{3}{2}\alpha_k + 2\alpha_k^2) x_0^k - \alpha_k^2 \mu (\sigma_k)$
which implies, after taking norms of both sides and using the fact that $\|\mu(\sigma_k)\|\leq 2$, 
 $ \dist_{k+1}\leq (1-\frac{3}{2}\alpha_k + 2\alpha_k^2) \dist_k + 2 \alpha_k^2.$
Then, by invoking classical results for the asymptotic behavior of non-negative sequences (see e.g. \cite[Appendix A.4.3]{Bertsekas15Book}, we get $\dist_{k+1}\to 0$. Theorem~\ref{theo-rate-quadratics} also shows global convergence of RR on this example.
} By a similar argument, it can be shown that SGD is also convergent to the optimal solution $x^*=0$ in mean-square, i.e. $\E \| y^j \|^2 \to 0$ (see also e.g. \cite{kushner2003stochastic}). Then, it follows from \eqref{eq-sgd-exam-error} and \eqref{grad-error-exam} that the cumulative gradient error of SGD for any cycle $k$ (defined as the cumulative sum $\sum_{j=(k-1)m}^{km-1} e_j)$) has zero expectation and $\Theta(1)$ variance whereas the gradient errors in RR are $E_k = \bigO(\alpha_k)$ with a typically non-zero expectation satisfying $\E (E_k) = \alpha_k (1-2x_0^k)$ and an asymptotically smaller variance $\bigO(\alpha_k^2)$ compared to SGD. In other words, the cycle gradient errors go to zero with probability one for RR whereas the gradient errors in SGD are typically bounded away from zero with a positive probability.  
Informally, this leads to a more accurate direction of descent for RR and is the main reason behind the faster convergence we demonstrate for RR compared to SGD in our analysis.

\begin{figure}\label{fig-1}
\begin{center}\hspace{-0.3in}
\vspace{-0.2in}
\hspace{-0.45in}
\includegraphics[width=1.2\linewidth]{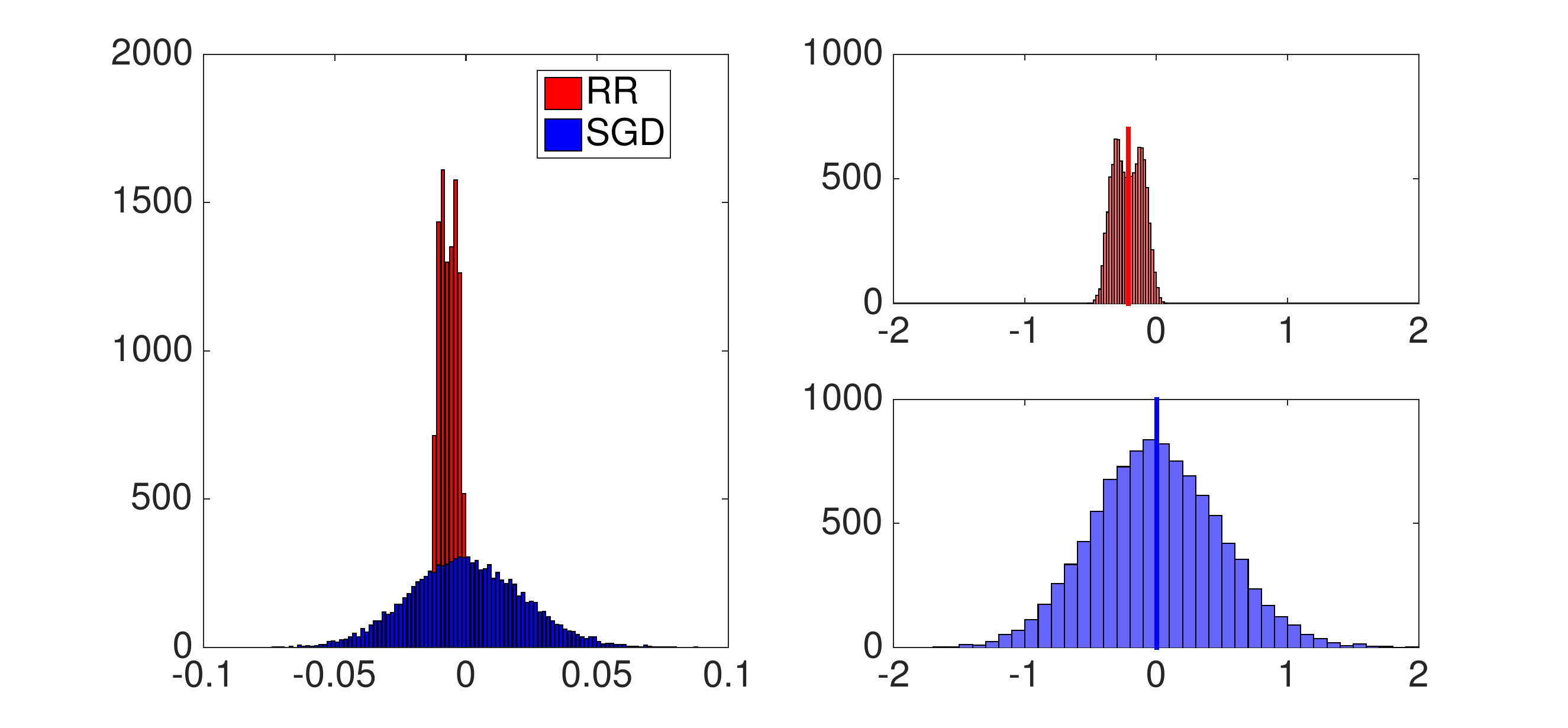}\label{quad-eq}
\caption{Left panel: Comparison of the histogram of the approximation error $\bar{x}_k - x^*$ of the averaged iterates for RR and SGD after $k=500$ cycles over $10000$ sample paths created for the Example \ref{exam-one} with $s=0.75$. Each sample path contains $1000$ gradient computations for both RR and SGD. Right, top panel: Histogram of the scaled approximation error $k^s (\bar{x}_k - x^*)$ for RR iterates which is concentrated around the vertical line in red. Right, bottom panel: Histogram of the scaled approximation error $k^{1/2} (\bar{x}_k - x^*)$ for SGD which has the shape of a standard normal distribution. The vertical blue line passing through the origin is the axis of symmetry for this distribution indicating that this distribution is centered.}
\end{center}
\end{figure}

We also observe that the cycle gradient error $E_k$ given by \eqref{grad-error-exam} consists of the sum of two terms: The first term is $\bigO(\alpha_k)$ and is independent over the cycles as the permutations $\sigma_k$ are independent and identically distributed whereas the second term is of smaller (second) order as $x_0^j \to 0$. We will show later in  Lemma~\ref{lem-martingale-variance} that such a decomposition can be obtained more generally when component functions are quadratics or they are smooth functions and will be a key step in the proof of Theorem~\ref{thm-averaged-rate-quad}. 

Figure \ref{fig-1} compares the RR and SGD algorithms with averaging in terms of the histogram of the error (distance of the averaged iterates to the optimal solution $x^*$). In other words, we compare the approximation errors $\bar{x}_k - x^*$ and $\bar{y}_{k} - x^*$ where where $\bar{y}_{k} := \frac{\sum_{j=0}^{mk-1} y^j}{mk} $ is the averaged SGD iterates after $k$ cycles (or equivalently $mk$ inner iterations). For a fair comparison, both algorithms are run with the same parameters using $k=500$ cycles over $10000$ sample paths created for the Example \eqref{exam-one} where $s=0.75$. The left panel in Figure \ref{fig-1} compares the histograms of $\bar{x}_k - x^*$ and $\bar{y}_{k} - x^*$ and shows that the approximation error $\bar{x}_k - x^*$ for RR is typically much smaller compared to that of SGD suggesting RR has a faster convergence rate. The top panel on the right illustrates that the scaled approximation error $k^s (\bar{x}_k - x^*)$ is concentrated around its mean (marked by the red line) suggesting $\bigO(1/k^s)$ convergence rate almost surely for the averaged RR iterates. On the other hand, the bottom panel on the right shows that the distribution of $k^{1/2} (\bar{y}_k - x^*)$ is approximately a standard normal distribution as predicted by the theory \cite{PolyakJuditsky92}, illustrating the $\bigO(1/k^{1/2})$ convergence rate of the averaged SGD iterates to the optimal solution $x^*$ in distribution. In Section \ref{sec:quad-case}, we will develop the first convergence theory for RR, establishing the $\bigO(1/k^s)$ convergence rate we observe in the numerical experiments and show that $k^s (\bar{x}_k - x^*)$ converges almost surely to a point for which we provide an explicit formula.

%
\end{example} 

\section{Quadratic component functions}\label{sec:quad-case}
We first consider quadratic component functions which allows an elegant analysis without the need to approximate higher order terms. We will show in Section \ref{sec:general-case} that the same line of analysis extends to smooth component function under a Lipschitz assumption on the Hessian matrices.
Let $f_i: \R^n \to \R$ be a quadratic function of the form
    \beq f_i(x) = \frac{1}{2} x_i^T P_i x - q_i^T x + r_i, \quad i=1,2,\dots, m,
    	\label{def-fi}
    \eeq
where $P_i $ is a symmetric $n\times n$ matrix, $q_i \in \R^n$ is a column vector and $r_i$ is a scalar. Note that $f_i$ has Lipschitz gradients, i.e., 
    $$ \| \nabla f_i(y) - \nabla f_i (z) \| \leq L_i \| y - z\|, \quad \forall y, z \in \R^n, $$
where $L_i = \| P_i \|$. It follows from the triangle inequality that $f$ has Lipschitz gradients with Lipschitz constant at most
	\beq  L := \sum_{i=1}^m L_i .
		\label{def-L}
	\eeq
Moreover, Assumption \ref{assump-sum-is-str-cvx} implies that the Hessian matrix of the sum satisfies 
$\nabla^2 f(x) = \sum_{i=1}^m \nabla^2 f_i(x) = \sum_{i=1}^m P_i \ge cI_n >0.$

\subsection{Convergence Rate} \label{quad-conv}

Our convergence analysis of RR builds on a recent upper bound for convergence rate of (deterministic) cyclic IG method (see  \cite{Gur2015IncGrad}), which can be generalized to hold for any fixed permutation $\sigma$ of $\{1,2,\dots,m\}$. This result implies an upper bound (for all sample paths) on the distance to the optimal solution of the iterates generated by RR, which is presented next.

\begin{theorem}\cite{Gur2015IncGrad} 
Let Assumption \ref{assump-sum-is-str-cvx} hold. Let $f_i(x)$ be a quadratic function of the form $f_i(x) = \frac{1}{2} x_i^T P_i x - q_i^T x + r_i$
where $P_i $ is a symmetric $n\times n$ matrix, $q_i \in \R^n$ is a column vector and $r_i$ is a scalar for $i=1,2,\dots,m$. Suppose Assumption \ref{assump-sum-is-str-cvx}  holds. Consider the iterates $\{x_0^k\}$ generated by the iterations \eqref{inner-update-wo} with a fixed order $\sigma$ and stepsize $\alpha_k = R/(k+1)^s$ where $R>0$ and  $s \in (1/2,1)$. Then\footnote{The original result in \cite{Gur2015IncGrad} was stated for $\sigma=\{1,2,\dots,m\}$ but here we translate this result into an arbitrary permutation $\sigma$ of $\{1,2,\dots,m\}$ by noting that processing the set of functions $\{f_1, f_2, \dots, f_m\}$ with order $\sigma$ is equivalent to processing the permuted functions $\{f_{\s_1}, f_{\s_2}, \dots, f_{\s_m}\}$ with order $\{1,2,\dots,m\}$.
},
    	\begin{align} \dist_k &\leq \frac{R \| \mu(\sigma)\| }{c}\frac{1}{k^s} + o(\frac{1}{k^s})  &\mbox{if}& \quad 1/2<s<1, \label{rate-prob-one-no-average}\\
           			  \dist_k  &\leq \frac{R^2 \|  \mu({\sigma}) \|}{ Rc - 1} \frac{1}{k} + o(\frac{1}{k})  &\mbox{if}& \quad s=1 \mbox{ and } Rc>1,
    	\end{align}
where $c$ is the strong convexity constant of the sum function $f(x)$ and
\beq
	\mu(\sigma) =  - \underset{1\leq i < j\leq m}{\sum} P_{\sigma(j)} \nabla f_{\sigma(i)}(x^*). \label{def-Msigma}
\eeq  	
	\label{theo-rate-quadratics} 
\end{theorem}
This theorem provides an upper bound on the rate with a rate constant $\mu(\sigma)$ that depends on the order \(\sigma\). 
Note that the best rate that IG with a fixed order $\sigma$ can attain in terms of upper bounds is $\bigO(1/k)$ and requires a stepsize $R/(k+1)$ with $R>1/c$ (see also \cite{Gur2015IncGrad} for the lower bound of $\Omega(1/k)$ for IG under some conditions).
We next provide some upper bounds on $\mu(\sigma)$. We define 
	\beqa \Gs :&=& \underset{1\leq i \leq m}{\sup} \| \nabla f_i(x^*)\|, 
	\label{def-Gs} \\
	 M_\Gamma :&=&~ \sup_{\sigma \in \Gamma} \|\mu(\sigma)\|.
	 \label{def-M-gamma}	
\eeqa
Using  $L_i = \| P_i \|$ for each $i$, it follows from the triangle inequality that
	\beq  \| \mu(\sigma) \| \leq M_\Gamma  \leq \sup_{\sigma \in \Gamma}  \underset{1\leq i < j\leq m}{\sum} L_{\sigma(j)} \Gs \leq \sum_{j=1}^m  (j-1) L_{\sigma(j)} \Gs \leq Lm\Gs
		\label{def-M-gamma-bound}
	\eeq
where $L$ is the Lipschitz constant of the gradient of $f$ defined by \eqref{def-L}. By replacing $\mu(\sigma)$ by $M_\Gamma$ in Theorem \ref{theo-rate-quadratics} one can get an upper bound on the worst-case convergence rate that applies to any choice of fixed order $\sigma$. Using a similar argument along the lines of the proof of Theorem \ref{theo-rate-quadratics} on the convergence rate of IG, it is straightforward to show that RR never performs any slower than this worst-case convergence rate which is the subject of the next result. The idea is to bound the stochastic $\dist_k$ sequence from above point-wise. The proof is a simple exercise and is omitted due to space considerations. 
\begin{corollary}\label{coro-deter-decay-rate-dist} Under the setting of Theorem \ref{theo-rate-quadratics}, if $\sigma$ is sampled uniformly at each cycle instead of being kept fixed, then
         \begin{align} \dist_k &\leq \frac{R M_\Gamma }{c}\frac{1}{k^s} + o(\frac{1}{k^s})  &\mbox{if}& \quad 1/2<s<1,\\
           			  \dist_k  &\leq \frac{R^2 M_\Gamma}{ Rc - 1} \frac{1}{k} + o(\frac{1}{k})  &\mbox{if}& \quad s=1 \mbox{ and } Rc>1,
    	\end{align}
with probability one where $M_\Gamma$ is deterministic and is defined by \eqref{def-M-gamma}.
\end{corollary}

Corollary~\ref{coro-deter-decay-rate-dist} provides a simple worst-case upper bound on the rate, however the rate constant $M_\Gamma = \sup_\sigma \|\mu(\sigma)\|$ is pessimistic and can be thought as a worst-case performance measure that holds for every sample path. 
One way to get better constants is to consider convergence in expectation, a weaker notion of convergence compared to almost sure convergence. In the next theorem, we show that $M_\Gamma$ can be improved to a typically much smaller constant $\|\bar{\mu}\|$ where \beq 
	\bar{\mu} := \E \big(\mu(\sigma_1)\big) =
			\frac{\sum_{\sigma \in \Gamma} \mu(\sigma)}{|\Gamma|}. \label{def-mu-bar}
\eeq can be thought as a measure of average performance over the choice of random permutations. 
\begin{theorem}\label{theo-expected-rate} Let $f_i(x)$ be a quadratic function of the form $f_i(x) = \frac{1}{2} x_i^T P_i x - q_i^T x + r_i,$
where $P_i $ is a symmetric $n\times n$ matrix, $q_i \in \R^n$ is a column vector and $r_i$ is a scalar for $i=1,2,\dots,m$. Suppose Assumption \ref{assump-sum-is-str-cvx} holds. Consider the iterates $\{x_0^k\}$ generated by the RR iterations \eqref{inner-update-wo} and stepsize $\alpha_k = R/(k+1)^s$ where $R>0$ and  $s \in (0,1]$. Then, 
	\begin{align} \E \left(\dist_k\right) &\leq \frac{R \| \bar{\mu} \| }{c}\frac{1}{k^s} + o(\frac{1}{k^s}) \quad  &\mbox{if}& \quad 1/2<s<1 \label{exp-dist-upper-bound-small-s}, \\
           \E \left(\dist_k\right)  &\leq \frac{R^2 \|  \bar{\mu} \|}{ Rc - 1} \frac{1}{k} + o(\frac{1}{k}) \quad &\mbox{if}& \quad s=1 \mbox{ and } Rc>1, \label{exp-dist-s-1}
    \end{align}
where the expectation is taken over the sequence of iterates, $\bar{\mu}$ is defined by \eqref{def-mu-bar}. 
\end{theorem} 

\begin{remark} A consequence of Lemma \ref{lem-computing-mu-bar} proved in the Appendix is that \beq 
  	\bar{\mu} =  \frac{1}{2} \sum_{i=1}^m P_i \nabla f_i(x^*)
\label{mu-bar-alternative}.
\eeq
where $\bar{\mu}$ is defined by \eqref{def-mu-bar}. By the triangle inequality, 
	$\| \bar{\mu} \| \leq \sum_{i=1}^m L_i \Gs = L\Gs $
where $\Gs$ is defined by \eqref{def-Gs}. This upper bound is $m$ times smaller than the previous upper bound on $M_\Gamma$ in \eqref{def-M-gamma-bound}. \rev{As an example, consider $s= 1$ with $Rc = 2$. In this case, $L = \bigO(m)$, $c = \bigO(m)$, $R = \bigO(1/m)$ and $\| \bar{\mu}\| = \bigO(m)$. We obtain from \eqref{exp-dist-s-1} that $\E \left(\dist_k\right) = O(\frac{1}{mk}) + o(1/k)$. Note that, the performance guarantee for IG from Corollary \ref{coro-deter-decay-rate-dist} is $\dist_k = \bigO(1/k) + o(1/k)$ which is also worse than RR by a factor of $O(m)$. For a fair comparison with the SGD method, we consider running the SGD iterations for $j = km$ iterations so that both RR and SGD methods have access to the same number of component gradients. In this case, expected distance to suboptimality for SGD with the recommended $O(1/j)$ stepsize is $\bigO(\frac{1}{\sqrt{j}})$ where the hidden constants are independent of $m$ (see e.g. \cite{PolyakJuditsky92,Moulines:2011vy}) which is worse than the $\bigO(1/j)$ guarantees for RR when $j$ is sufficiently large. These bounds show that when $m$ is small and $k$ is large, IG could outperform SGD in theory, however SGD (and RR) are more suitable for applications when $m$ is large and will admit better bounds compared to IG if $m$ is large enough for a given $k$ fixed\footnote{We note however that SGD upper bounds are in expectation whereas IG results are deterministic which is a stronger notion of convergence.}}  
\end{remark}


It is also natural to ask what would happen to the rate constants and to the rate if one would take stepsize $\alpha_k = \Theta(1/k^s)$ and apply (Polyak-Ruppert) averaging to the RR iterates, especially given the fact that $\bigO(1/k^s)$ stepsize used in averaging does not require adjustment of the parameter $R$ to the strong convexity level. More generally, one could consider $q$-suffix averaging. In the next section, we show that for the averaged RR iterates, similar upper bounds in \eqref{exp-dist-upper-bound-small-s} hold not only in expectation but also in probability. Another benefit of averaging is that it leads to not only upper bounds but also lower bounds which can then be leveraged to accelerate RR further as we will show in Section \ref{sec:accelerated-RR}. 


\subsection{Convergence rate with averaging} The following theorem characterizes the rate of convergence of the averages of iterates generated by RR. Part $(i)$ and $(ii)$ of this theorem show that $q$-suffix averages of the RR iterates converge at rate $1/k^s$ to the optimal solution almost surely with a stepsize $\Theta(1/k^s)$ for $s \in (1/2,1)$. By gradient Lipschitzness, this translates into a rate of $\Theta(1/k^{2s})$ for the suboptimality of the objective value.  The result is based on decoupling the cycle gradient errors $E_k$ into a $\Theta(\alpha_k)$ term independent over the cycles and another $\bigO(\alpha_k^2)$ term that becomes negligible in the limit. Part $(iii)$ is a high-probability convergence rate estimate for the approximation error $\xqk - x^*$. The approximation error consists of two terms, the first term $b_{q,k}$ which we call the ``bias" term is deterministic and decays like $\sim 1/k^s$. It comes from the expected value of the independent part of the gradient cycle errors which may be different than zero. The second part is on the order of $1/k$ for $0<q<1$ (and $\log k/k$ when $q=1$) and it is based on the Azuma-Hoeffding inequality for martingale concentration. Finally, part $(iv)$ is on estimating the bias term $b_{q,k}$ with another quantity $\hat{b}_{q,k}$. It shows that by subtracting the estimated bias from the averaged iterates, we can approximate the optimal solution $x^*$ up to an $\bigO(1/k)$ error in distances or equivalently up to an $\bigO(1/k^2)$  error in the suboptimality of the objective value. In Section \ref{sec:accelerated-RR}, this result will be fundamental for  Algorithm \ref{fastrs} that accelerates the convergence of RR from $\Theta(1/k^{2s})$ to $\bigO(1/k^2)$ with high probability in the suboptimality of the objective value.


\begin{theorem}\label{thm-averaged-rate-quad}  Let $f_i(x)$ be a quadratic function of the form
	$$f_i(x) = \frac{1}{2} x_i^T P_i x - q_i^T x + r_i$$ 
where $P_i $ is a symmetric $n\times n$ matrix, $q_i \in \R^n$ is a column vector and $r_i$ is a scalar for $i=1,2,\dots,m$. Consider the $q$-suffix averages $\xqk$ of the RR iterates generated by the iterations \eqref{inner-update-wo} with stepsize $\alpha_k = \frac{R}{(k+1)^s}$ where $R>0$ and $s \in (\frac{1}{2},1)$. Suppose that Assumption \ref{assump-sum-is-str-cvx} holds. Then the following statements are true:
\begin{enumerate}
\item [$(i)$] For any $0 < q \leq 1$, the $q$-suffix averaged stepsize $\aqk$ defined in \eqref{def-averaged-stepsize} satisfies
    \beq\aqk = \frac{\aqs}{k^s} + \bigO(\frac{1}{k}) \quad \mbox{where} \quad \aqs = \frac{1-(1-q)^{1-s}}{q(1-s)} R.\label{def-aqk}
    \eeq
 \item [$(ii)$] For any $0 < q \leq 1$, we have 
    \beq \lim_{k \to \infty} \frac{\bar{x}_{q,k} - x^*}{\aqk} =  -H_*^{-1} \bar{\mu}\quad a.s. \label{limit-main-result}
    \eeq
where $\bar{\mu}$ is given by \eqref{mu-bar-alternative}, i.e., the normalized error $(\bar{x}_{q,k} - x^*)/\aqk$  converges to the constant vector $-H_*^{-1} \bar{\mu}$ almost surely where $\Hs=\sum_{i=1}^m P_i$ is the Hessian matrix at the optimal solution and $\bar{\mu}$ is given by \eqref{mu-bar-alternative}.
Then, from part $(i)$,  
 \beq \lim_{k \to \infty} k^s (\bar{x}_{q,k} - x^*) =  -\aqs H_*^{-1} \bar{\mu} \quad a.s.\label{lim-rate-aver-quad}
 \eeq
Hence, the $q$-suffix averaged iterates $\xqk$ converge to the optimal solution $x^*$ with rate $1/k^s$ almost surely. 
   
 \item [$(iii)$] With probability at least $1-\delta$, we have
	 $$ \bar{x}_{q,k} -  x^*  =  b_{q,k} + \bigO \bigg(\frac{\sqrt{\log(1/\delta)}}{k} \bigg) +  \begin{cases} \bigO \big(\frac{\log k}{k} \big)  & \mbox{if} \quad  q=1 \\ \bigO(\frac{1}{k}\big) & \mbox{if} \quad  0<q<1, \end{cases} $$ where
	 \beq b_{q,k} = -\bar{\alpha}_{q,k} H_*^{-1}  \bar{\mu} 
	 		\label{def-bqk}
	 \eeq 
is deterministic, $\bar{\mu}$ is given by $\eqref{mu-bar-alternative}$ and $\bar{\alpha}_{q,k}$ is the averaged stepsize defined in \eqref{def-averaged-stepsize}. The constants hidden by $\bigO(\cdot)$ depend only on $\Gs, L, m, R, c, q$ and $s$.
	 \item [$(iv)$] Let \beq  \hat{b}_{q,k} = -\bar{\alpha}_{q,k}\bigg[\sum_{i=1}^{m} P_{\ski}\bigg]^{-1} \sum_{i=1}^{m} P_{\ski} \nabla f_{\ski} (x_{i-1}^k)/2. \label{def-hat-bqk}
	 \eeq
where $\bar{\alpha}_{q,k}$ is the averaged stepsize defined in \eqref{def-averaged-stepsize}. Then, 	 
	 $\hat{b}_{q,k}  = b_{q,k} + \bigO(\alpha_k^2). $
It follows from part $(ii)$ that with probability at least $1-\delta$,
	  $$ (\bar{x}_{q,k} - \hat{b}_{q,k}) -  x^*  = \bigO \bigg(\frac{\sqrt{\log(1/\delta)}}{k} \bigg) +  \begin{cases} \bigO \big(\frac{\log k}{k} \big)  & \mbox{if} \quad  q=1 \\ \bigO(\frac{1}{k}\big) & \mbox{if} \quad  0<q<1. \end{cases} $$ 
\end{enumerate} 
 
 \end{theorem} 

 \begin{proof} 
	\begin{itemize} 
	\item [$(i)$] As the stepsize sequence is monotonically decreasing, we have the bounds 
   \beqas \int_{(1-q)k}^k  \frac{R}{(x+2)^s} dx \leq  \sum_{j=(1-q)k}^k \alpha_j = \sum_{j=(1-q)k}^k \frac{R}{(k+1)^s} 
   			 &\leq& R + \int_{(1-q)k}^{k-1}  \frac{R}{(x+1)^s} dx. 
   \eeqas
Dividing each term by $qk$, after a straightforward integration we obtain
  $$ \aqk = \frac{ k^{1-s} - \big((1-q)k+1\big)^{1-s} + \bigO(1)}{(1-s)qk} R  = \frac{a_q(s)}{k^s} + \bigO(\frac{1}{k}). 
   $$
 which completes the proof.  
 			\item [$(ii)$] Taking the $q$-suffix averages of both sides of \eqref{averaged-iter-randomized}, we obtain 
\beq 
	I_{q,k} := \frac{\sum_{j=(1-q)k}^{k-1} {(x_0^j - x_{0}^{j+1})}{\alpha_j^{-1}}}{qk} = \frac{ \sum_{j=(1-q)k}^{k-1}   \nabla f(x_0^j)+ E_j}{qk}. \label{averaged-iter-random-2}
   \eeq 
As $f$ is a quadratic, the first order Taylor series for the gradient of $f$ is exact:
  \beq \nabla f(x_0^j) = H_* (x_0^j-x^*).
  		\label{grad-as-hessian-prod-quadratics}
  \eeq
 Therefore, \eqref{averaged-iter-random-2} becomes  $
I_{q,k} 
= \frac{ \sum_{j=(1-q)k}^{k-1}   H_* (x_0^j - x^*)   + E_j }{qk} $
which is equivalent to
\beqa
                    I_{q,k} &=& H_*(\bar{x}_{q,k} - x^*) + \frac{\sum_{j=(1-q)k}^{k-1} E_j}{qk} = H_*(\bar{x}_{q,k} - x^*) + \bar{\alpha}_{q,k}  Y_{q,k}\label{iterate-eq-randomized}
   \eeqa 
where $Y_{q,k}$ is defined as
    \beq Y_{q,k} := \frac{1}{\bar{\alpha}_{q,k}} \frac{\sum_{j=(1-q)k}^{k-1} E_j }{qk } =  \frac{\sum_{j=(1-q)k}^{k-1} E_j}{\sum_{j=(1-q)k}^{k-1} \alpha_j}	
		\label{def-Yqk}    
     \eeq
and can be interpreted as the ($q$-suffix) averaged gradient error sequence $E_j$  normalized by the ($q$-suffix) averaged stepsize sequence $\alpha_j$. Since $H_*$ is invertible by the strong convexity of $f$ (see \eqref{def-Hs}), we can rewrite \eqref{iterate-eq-randomized} as
\beqa  \bar{x}_{q,k} - x^* &=& 
                     - H_*^{-1} \aqk \Yqk  + H_*^{-1} \Iqk  
            = - H_*^{-1} \bar{\alpha}_{q,k} Y_{q,k}  +  \begin{cases} \bigO(\frac{1}{k})  & \mbox{if } \quad 0<q<1 \\
\bigO( \frac{\log k}{k})  & \mbox{if } \quad q = 1.
\end{cases} \label{xk-averaged-asymptotics-2}
   \eeqa              
where we used the inequality $\| H_*^{-1}\| \leq 1/c $ implied by \eqref{def-Hs} and Lemma \ref{lem-I-lk} from the appendix to provide an upper bound for the second term in the first equality. Note that, as a consequence of Lemma \ref{lem-I-lk}, $\bigO(\cdot)$ notation above hides a constant that depends only on the parameters $\Gs, L,c, m, R, s, q$ and also $\dist_0$ when $q=1$. Then, dividing both sides of \eqref{xk-averaged-asymptotics-2} by $\aqk$, taking limits as $k$ goes to infinity,  using part $(i)$ on the asymptotic behavior of $\aqk$ and the fact that $Y_{q,k} \to \bar{\mu}$ a.s. from Lemma \ref{lem-martingale-variance}, we obtain the claimed result.
%
%
%
%

 \item [$(iii)$] By parts $(i)$ and $(iii)$ of Lemma \ref{lem-martingale-variance} from the appendix that relates the gradient error sequence $E_j$ to a sequence of i.i.d. variables $\mu(\sigma_j)$, for $0<q\leq 1$,  
		\beq Y_{q,k} = \frac{\sum_{j=(1-q)k}^{k-1} E_j}{\sum_{j=(1-q)k}^{k-1} \alpha_j} =  \frac{\sum_{i=(1-q)k}^{k-1} \alpha_j \mu(\sigma_j) + \bigO(\alpha_j^2)}{\sum_{j=(1-q)k}^{k-1} \alpha_j}. 	 
			\label{Yqk-helper}  
     \eeq 
We first give a proof for $q=1$, the proof for the remaining $q \in (0,1)$ case will be similar. Assume $q=1$. Plugging $q=1$ and \eqref{Yqk-helper} into \eqref{xk-averaged-asymptotics-2}, we obtain
     \beqa \bar{x}_{1,k} - x^* 
     &=&\bigO(\frac{\log k}{k})  - H_*^{-1} \aonek \Yonek \nonumber \\
     &=& \bigO(\frac{\log k}{k}) - H_*^{-1} \bigg(   \frac{\sum_{j=0}^{k-1} \alpha_j \big(\mu(\s_j) - \bar{\mu}\big)}{k} +  \frac{\sum_{j=0}^{k-1}  \alpha_j \bar{\mu} + \bigO(\alpha_j^2) }{k}  \bigg) \nonumber \\
     						 &=& b_{1,k}  +  \bigO(\frac{\log k}{k})  - H_*^{-1}   \frac{\sum_{j=0}^{k-1}  \alpha_j (\mu(\s_j) - \bar{\mu})}{k} - H_*^{-1} \sum_{j=0}^{k-1}  \frac{\bigO(\alpha_j^2)}{k} \nonumber \\     
     						    &=& b_{1,k}  +  \bigO(\frac{\log k}{k})  - H_*^{-1}   \frac{\sum_{j=0}^{k-1}  \alpha_j (\mu(\s_j) - \bar{\mu})}{k} \label{xqk-high-proba} 
     \eeqa    
where $\bonek$ is defined by \eqref{def-bqk} and we used in the last step the fact that for $s>1/2$ 
	\beq \sum_{j=0}^\infty \alpha_j^2 = \sum_{j=1}^\infty \frac{R^2}{j^{2s}} = R^2 \zeta(2s) < \infty
		\label{square-summable-step}	
	\eeq 
where $\zeta(\cdot)$ is the Riemann-Zeta function. We now study the asymptotic behavior of the last summation term in \eqref{xqk-high-proba} by introducing the process
	$S_{1,k} = \sum_{j=0}^{k-1} Z_j$, where $Z_j := \alpha_j (\mu(\s_j) - \bar{\mu})$ and $k\geq 0$ with the convention that $S_{1,0} = 0$. Equipped with this definition, \eqref{xqk-high-proba} becomes
  \beq 
  		 \xonek - \xs = b_{1,k}  +  \bigO(\frac{\log k}{k})  - H_*^{-1}   \frac{S_{1,k}}{k}.  \label{xk-biased-high-probability}
  \eeq		  
The random variables $Z_j$ are independent, centered and have an identical distribution up to the scaling factor $\alpha_j$. Therefore,  $S_{1,k}$ is a sum of  centered random variables satisfying:
    \beqa \begin{scriptsize} \| S_{1,k} - S_{1,k-1} \| =\big \|\alpha_{k-1}\big(\mu(\sigma_{k-1})-\bar{\mu}\big)\big\| 
    \leq \gamma_{k-1}:=\alpha_{k-1}LmG_* \label{marting-diff-as-bound}
    \end{scriptsize}
    \eeqa 
where we used \eqref{def-mu-sk-upper-bound} in the last inequality (see also Lemma \ref{lem-computing-mu-bar}). Then, by the Azuma-Hoeffding inequality, for every $t>0$, 
  \beqas \P \bigg( \big\|\frac{S_{1,k}}{k}\big\| > \frac{t}{k}\bigg) &\leq& 2 \exp \bigg( - \frac{t^2}{2\sum_{j=0}^{k-1} \gamma_j^2} \bigg) =  2 \exp \big( - \frac{t^2}{ \beta }\big) 
  \eeqas
where $\beta = 2\sum_{j=0}^{\infty} \gamma_j^2 < \infty$ as $\alpha_j$ is square-summable (see \eqref{square-summable-step}). Note that $\beta$ depends only on $\Gs, L,m$ and the stepsize parameters $R$ and $s$. It is easy to see that selecting $t \geq t_\delta = \sqrt{\beta  \log(2/\delta)}$ makes the right-hand side $\leq \delta$.  Therefore for any $\delta>0$, with probability at least $1-\delta$, 
		\beq \bigl \| \frac{S_{1,k}}{k} \bigr \| \leq \frac{\sqrt{\beta  \log(2/\delta)}}{k}
			\label{average-sum-bound-whp}
		\eeq
which if inserted into the expression \eqref{xk-biased-high-probability} completes the proof for the $q=1$ case. For  $0<q<1$ case, the same line of reasoning applies except that we replace $b_{1,k}$ with $b_{q,k}$ and we can improve the $\bigO( \log k /k)$ term in the expression \eqref{xk-biased-high-probability} to $\bigO(1/k)$, this is justified by \eqref{xk-averaged-asymptotics-2}. Then, this leads to 
		 \beq 
  		 \xqk - \xs = b_{q,k}  +  \bigO(\frac{1}{k})  - H_*^{-1}   \frac{S_{q,k}}{qk}  \label{xk-biased-high-probability-3}
  \eeq
where $S_{q,k} := \sum_{j=(1-q)k}^{k-1} Z_j = S_{1,k} - S_{1,(1-q)k}$ is the $q$-suffix cumulative sum (cumulative sum of the last $qk$ terms) of the sequence $Z_{k}$. Then using \eqref{average-sum-bound-whp}, with probability at least $1-\delta$,
\beq \bigl \| \frac{S_{q,k}}{k} \bigr \| \leq \| \frac{S_{1,k}}{k} \bigr \|  + \| \frac{S_{1,(1-q)k}}{k} \bigr \|  \leq \frac{2 t_\delta}{k}.
			\label{average-sum-bound-whp-3}
		\eeq
Plugging this high probability bound into \eqref{xk-biased-high-probability-3}, we conclude. 		
	\item [$(iv)$] By Lemma \ref{lem-random-tail-bound}, we have $\underset{1\leq i < m}{\max} \|x_{i-1}^k - x^*\| = \bigO(\alpha_k)$. Therefore,
	 \beqa \scriptsize \| \nabla f_{\ski} (x_{i-1}^k)  - \nabla f_{\ski} (x^*) \|  = \bigO(\alpha_k) \label{grad-diff-quadratic}
	 \eeqa
for any $i=1,2,\dots,m$. As a consequence, 	
	\beqas \hat{b}_{q,k} &=& -\bar{\alpha}_{q,k} H_*^{-1} \sum_{i=1}^{m} P_{\ski}  \bigg(\nabla f_{\ski} (x^*) + \bigO(\alpha_k)\bigg)  \\
	                         &=& - \bar{\alpha}_{q,k} H_*^{-1} \sum_{j=1}^{m} P_{j} \nabla f_j (x^*)   + \bigO(\alpha_k^2)   
	                         = b_{q,k} + \bigO(\alpha_k^2)	                         	                        	                         
	\eeqas		 
where in the second equality we use the fact that $\bar{\alpha}_{q,k}= \bigO(1/k^s) = \bigO(\alpha_k)$ implied by part $(i)$.
\end{itemize}
\end{proof}

\section{Extension to smooth component functions}\label{sec:general-case}
Extending our results to more general smooth functions requires obtaining similar bounds for the cycle gradient errors which depend on the gradients and Hessian matrices of the component functions along the inner iterates. In order to be able to control the change of gradients and Hessian matrices along the iterates, we introduce the following assumption which has also been used to analyze SGD \cite{Moulines:2011vy}.

\begin{assumption}\label{Lip-Hessian} The functions $f_i$ are convex on $\R^n$ and have Lipschitz continuous second derivatives, i.e. there exists a constant $\hlip_i$ such that 
   $$ \| \nabla^2  f_i(x) - \nabla^2 f_i(y) \| \leq \hlip_i \| x - y \|, \quad \forall x, \forall y \in \R^n, $$
for $i=1,2,\dots,m$. 
\end{assumption}

Under this assumption, by the triangle inequality, $\nabla^2 f(\cdot)$ is also Lipschitz with constant $  \hlip:= \sum_{i=1}^m \hlip_i. 
	\label{def-lip-hess-cont}
$
When the component functions are quadratics, we have the special case with $U=U_i = 0$. We will now see how this assumption makes it possible to control the change of gradients of the component functions. Smooth functions $f$ with Lipschitz Hessians are quadratic-like in the sense that the first-order Taylor approximation to the gradient of $f$ is almost affine (with a quadratic term controlled by the parameter $U$) satisfying
	\beq \nabla f(x) = \nabla f(\xs) + H_* (x - \xs) + \eta, \quad \| \eta \| \leq \frac{U}{2}\| x- \xs\|^2, \quad \forall x.
		\label{taylor-of-lip-Hessian}
	\eeq
(see e.g. \cite[Section 1.3]{Gould2003Book})	
The analysis of Theorem \ref{thm-averaged-rate-quad} (and Lemma \ref{lem-martingale-variance} it builds upon) considers the $U=0$ case (see e.g. \eqref{grad-as-hessian-prod-quadratics} and \eqref{grad-diff-quadratic}) applying a first-order Taylor approximation to the gradient of the component functions at $x=x_0^k$ where $\|x-x^*\| = \|x_0^k-x^*\| = \bigO(\alpha_k)$ by Lemma \ref{lem-random-tail-bound}. Therefore, when $U\neq 0$, an extra correction term $\eta = \bigO(\alpha_k^2)$ needs to be added to the analysis. However, we show in the next theorem that this correction term does not cause a slow down in the convergence rate (in terms of dependency in $k$) compared to the quadratic case because the $q$-suffix averages of this $\bigO(\alpha_k^2)$ correction term decays like $\bigO(1/k)$.\footnote{This is due to the fact that the sequence $\alpha_k^2$ is summable when $s>1/2$.} 

We will also need one more technical assumption that appeared in a number of papers in the literature for analyzing incremental methods to rule out the case that the iterates diverge to infinity. In particular, this assumption is made in \cite{Gur2015IncGrad} for generalizing Theorem \ref{theo-rate-quadratics} on the rate of deterministic IG from quadratic functions to general smooth functions which we will be referring to. 

\begin{assumption}\label{assum-iters-bdd} Iterates $\{x_j^k\}_{j,k}$ generated are uniformly bounded, i.e. there exists a non-empty compact Euclidean ball $\mathcal{X} \subset \R^n$ that contains all the iterates a.s.\footnote{Note that if this assumption holds and if $f_i$ is three-times continuously differentiable on the compact set $\mathcal{X}$, then the third-order derivatives are bounded and Assumption \ref{Lip-Hessian} holds.}\end{assumption}

Equipped with these two assumptions, all the results of Theorem \ref{thm-averaged-rate-quad} extend naturally with minor modifications. In particular, $P_i$ (which is a constant Hessian matrix in the setting of Theorem \ref{thm-averaged-rate-quad}) needs to be replaced by $\nabla^2 f_i (\xs)$ or $\nabla^2 f_i (x_{i-1}^k)$ depending on the context.  
\begin{theorem}\label{thm-averaged-random-rate} Consider the RR iterations given by \eqref{inner-update-wo} with stepsize $\alpha_k = \frac{R}{(k+1)^s}$ where $R>0$ and $s \in (\frac{1}{2},1)$. Suppose that Assumptions \ref{assump-sum-is-str-cvx}, \ref{Lip-Hessian} and \ref{assum-iters-bdd} hold. Then the following statements are true:
\begin{enumerate}
 \item [$(i)$] For any $0 < q \leq 1$,  $ \lim_{k \to \infty} k^s (\bar{x}_{q,k} - x^*) = -\aqs H_*^{-1} \mub \quad a.s.$
where $\Hs=\nabla^2 f(x^*)$ is the Hessian matrix at the optimal solution, $a_q(s)$ is defined by \eqref{def-aqk}
and 
\beq      
    \mub :=  \frac{1}{2} \sum_{i=1}^m \nabla^2 f_i(x^*) \nabla f_i(x^*). 
 \eeq
   
 \item [$(ii)$] With probability at least $1-\delta$, we have
	 $$ \bar{x}_{q,k} -  x^*  =  r_{q,k} + \bigO \bigg(\frac{\sqrt{\log(1/\delta)}}{k} \bigg) +  \begin{cases} \bigO \big(\frac{\log k}{k} \big)  & \mbox{if} \quad  q=1 \\ \bigO(\frac{1}{k}\big) & \mbox{if} \quad  0<q<1, \end{cases} $$ where \beq r_{q,k} = - \bar{\alpha}_{q,k}H_*^{-1}  \mub 
	 		\label{def-bqk-2}
	 \eeq 
is deterministic. The constants hidden by $\bigO(\cdot)$ depend only on $\Gs, L, m, R, c, q, s$ and $\hlip$.
 	 \item [$(iii)$] Let $$ \hat{r}_{q,k} = -\bar{\alpha}_{q,k}\bigg[ \sum_{i=1}^{m} \nabla^2 f_{\ski}(x_{i-1}^k)\bigg]^{-1} \sum_{i=1}^{m}\nabla^2 f_{\ski} (x_{i-1}^k)  \nabla f_{\ski} (x_{i-1}^k)/2.$$
Then, $\hat{r}_{q,k}  = r_{q,k} + \bigO(\alpha_k^2).$
It follows from part $(ii)$ that with probability at least $1-\delta$,
	  $$ (\bar{x}_{q,k} - \hat{r}_{q,k}) -  x^*  = \bigO \bigg(\frac{\sqrt{\log(1/\delta)}}{k} \bigg) +  \begin{cases} \bigO \big(\frac{\log k}{k} \big)  & \mbox{if} \quad  q=1 \\ \bigO(\frac{1}{k}\big) & \mbox{if} \quad  0<q<1. \end{cases} $$ 
\end{enumerate} 
 \end{theorem} 
        
\begin{proof} The proof techniques of Theorem \ref{thm-averaged-rate-quad} applies directly except that the Taylor approximation for the gradients of the component functions will have an extra term compared to the proof of Theorem \ref{thm-averaged-rate-quad} (see also \eqref{taylor-of-lip-Hessian}). Also, instead of Lemmas \ref{lem-I-lk} and \ref{lem-martingale-variance} that apply to only quadratic functions, their extensions Lemmas \ref{lem-I-Ik-2} and \ref{lem-martingale-variance-2} given in the appendix are used in the proof. For the sake of completeness, besides these changes, we also give an overview of the main modifications required for each part of the proof: 
\begin{itemize}
	\item [$(i)$] The expression \eqref{grad-as-hessian-prod-quadratics} for the gradient should be modified to include an extra error term $\eta_j$ of the form
\beq \nabla f(x_0^j) = H_* (x_0^j-x^*) + \eta_j, \quad \| \eta_j \| \leq \frac{\hlip}{2} \| x_0^j - x^* \|^2 \label{lip-hessian-iter-bound}
\eeq 
By Lemma \ref{lem-random-tail-bound-2}, $\sum_j \eta_j \leq \frac{\hlip}{2} \|\| x_0^j - x^* \|^2 = \bigO(\alpha_j^2)$ therefore the sequence $\eta_j$ is summable and if averaged decays like $\bigO(1/k)$ without degrading the convergence rate except possibly the constants hidden by $\bigO(\cdot)$. 
	\item [$(ii)$] The same proof applies by invoking Lemma \ref{lem-martingale-variance-2} in lieu of Lemma \ref{lem-martingale-variance}. 
	\item [$(iii)$] Instead of Lemma \ref{lem-random-tail-bound}, we use Lemma \ref{lem-random-tail-bound-2}. The expression \eqref{grad-diff-quadratic} on the difference of gradients needs to be adjusted as
	\beq \| \nabla f_{\ski} (x_{i-1}^k)  - \nabla f_{\ski} (x^*) - \nabla^2 f_{\ski} (x^*) (x_{i-1}^k - x^*) \| \leq \frac{U}{2}\| x_i^k - x^*\|^2. \label{grad-diff-quadratic-2}
	 \eeq 	 
The right-hand side is still $\bigO(\alpha_k^2)$ by an application of Lemma \ref{lem-random-tail-bound-2} therefore the rest of the proof applies.
\end{itemize}
%
%
\end{proof}

\section{An RR algorithm with bias removal}\label{sec:accelerated-RR}
Part $(iii)$ of Theorem \ref{thm-averaged-random-rate} (see also part $(iii)$ of Theorem \ref{thm-averaged-rate-quad}) shows that if the estimate of the bias term $\hat{r}_{q,k}$ given by \eqref{def-bqk-2} is subtracted from the $q$-suffix averaged RS iterates, then the distance to the optimal solution of the $q$-suffix averaged iterates becomes on the order of $\bigO(1/k)$ for $0<q<1$ and on the order of $\bigO(\log k/k)$ for $q=1$ with high probability. By strong convexity, this translates into a rate of $\mathcal{\tilde{O}}(1/k^2)$ in the suboptimality of the objective values (where $\mathcal{\tilde{O}}$ ignores the logarithmic terms in $k$ appearing when $q=1$.). We call this ``subtraction operation", \textit{bias removal}. Algorithm DRR describes how this can be implemented. In a practical implementation, the number of cycles can be fixed in advance to a certain number $K$, and the estimation of the bias can be done only once at the last ($K$-th) cycle (see Step $(ii)$ of Algorithm \ref{fastrs}) and then can be subtracted from the averaged iterates. 

\begin{algorithm}[h!]\label{alg:fastrs}
\small
\caption{\textbf{D}e-biased \textbf{Random} \textbf{R}eshuffling (DRR) }\label{fastrs}
\textbf{Input:} Initial point $x_0^0 \in \R^n$, number of cycles $K \in \N$, suffix averaging parameter $q \in (0,1]$, stepsize parameters $R>0$ and $s\in (1/2,1)$. 

\textbf{Initialization:} $\bar{x}_{1,0}=0 \in \R^n$, $\hat{v}_0=0 \in \R^n$, $\bar{\alpha}_{1,0}=0 \in \R$, $\hat{H}_0 = 0 \in \R^{n\times n}$.  
\begin{enumerate}

\item For each cycle $k=0,1,2,\dots, {K-1}$:
\begin{enumerate}
	\item Inner iteration.
	    \begin{itemize} 
			\item [$(i)$] Pick a permutation $\sigma_k$ of $\{1, \dots , m\}$ uniformly at random. 
	 
	 		\item [$(ii)$] For $i=1,2,\dots,m$: 
	 
	 	\quad \quad Compute $x_i^k$ by:  $ x_{i}^k ={x_{i-1}^k - \alpha_k \nabla f_{\sigma_k(i)} (x_{i-1}^k)}, \quad \alpha_k = \frac{R}{(k+1)^s}. $\\
        // Precompute for the bias estimation only for the last cycle \\         	 
	 	 \quad \quad If $k=K-1$, compute $\hat{v}_i$ and $\hat{H}_i$ by : 	   
	 		\beqas 
	 			\hat{v}_{i} = \hat{v}_{i-1} + \nabla^2 f_{\ski}(x_{i-1}^k) \nabla f_{\ski} (x_{i-1}^k)/2, \quad
	 			 \hat{H}_i = \hat{H}_{i-1} + \nabla^2 f_{\ski}(x_{i-1}^k)	
	 		\eeqas
        	\item [$(iii)$] Set outer iterate: $x_0^{k+1} = x_m^k$.
     \end{itemize}
     \item Update the simple average of the iterates and the stepsize:
     		\beqas \bar{x}_{1, k+1}=  \frac{k}{k+1} \bar{x}_{1,k} + \frac{1}{k+1} x_0^{k}, \quad
     		           \bar{\alpha}_{1,k+1} = \frac{k}{k+1} \bar{\alpha}_{1,k} + \frac{1}{k+1} \alpha_k  
     		\eeqas   
\end{enumerate}

\item If $q \in (0,1)$, compute $q$-suffix averages from the simple averages: 
   		 		\beqs \bar{x}_{q,K} = \frac{
   \bar{x}_{1,K} - q \bar{x}_{1,(1-q)K} }{1-q}, \quad \bar{\alpha}_{q,K} = \frac{
   \bar{\alpha}_{1,K} - q \bar{\alpha}_{1,(1-q)K} }{1-q}.
   				\eeqs

\item Estimate the bias by the formula \eqref{def-hat-bqk} : $ \hat{b}_{q,K} = - \bar{\alpha}_{q,K} \hat{H}_m^{-1} \hat{v}_m$ in the last cycle.
\end{enumerate}
\textbf{Output:} $\bar{x}_{q,K} - \hat{b}_{q,K}$.  	
\end{algorithm}

The bias removal of the DRR algorithm requires an $n\times n$ matrix inversion which requires $\approx n^3$ arithmetic operations (if there is more structure on the Hessian of $f_i$ such as low-rankness or sparsity this could be improved to $\approx n^2$), but accelerates the convergence with high-probability. For small or moderate $n$, this could be done efficiently and incrementally processing the functions one at a time; however for large $n$ this may be impractical or infeasible limiting the applicability of this method. Nevertheless, the expensive matrix inversion step does not need to be done at every cycle, it suffices to do it only once at the end of the last cycle. Figure \ref{fig-3} compares the performance of SGD, RR and DRR  methods in terms of the histogram of the distance to the optimal solution (left panel) and suboptimality of the objective function (right panel) on a randomly generated quadratic example with a dense Hessian matrix with parameters $m=50$, $n=20$. For a fair comparison, we run all the algorithms with the same amount of CPU time. In particular, in Figure \ref{fig-3} we run DRR for 0.5 seconds including the bias correction step, and run RR and SGD for the same amount of time. We observe that SGD is consistently performing the worst, whereas DRR leads often to a better solution than RR both in terms of distances to the optimal solution and suboptimality. Figure \ref{fig-4} repeats the experiment with 5 seconds, we see a clearer separation between the histograms of the RR method and the De-biased RR method. We see similar results when we run the algorithms for different amount of times. These results show that the asymptotic performance would get better if one removes the bias term and typically we need more cycles for the bias correction term to be effective. The results also illustrate the results of Theorem \ref{thm-averaged-rate-quad} and \ref{thm-averaged-random-rate} on the biasedness of the RR iterations in the sense that asymptotically an improvement can be obtained by subtracting the bias. 

\section{Conclusion}\label{sec:conclusion}

We analyzed the random reshuffling (RR) method for minimizing a finite sum of convex component functions. When the objective function is strongly convex and the component functions are smooth, averaged RR iterates converge at rate $\sim1/k^{s}$ to the optimal solution almost surely (which translates into a rate of $1/{k^{2s}}$ in the suboptimality of the objective value) for a diminishing stepsize $\alpha_k = \Theta(1/k^s)$ with $s \in (1/2,1)$. This is faster than SGD's $\Omega(\frac{1}{k})$ rate. Viewing RR as a gradient descent method with random gradient errors, this result builds on first showing that gradient errors $E_k$ satisfying $ E_k = \bigO(\alpha_k)$ and then relating the gradient error sequence to an i.i.d sequence to which martingale theory is applicable. Note that the gradient errors in SGD are larger with a $\bigO(1)$ variance, which leads to a less accurate gradient descent direction. Beyond RR and SGD comparison, these results also give insight into the fast convergence properties of without-replacement sampling strategies  compared to with-replacement sampling strategies. 

After characterizing the convergence rate of RR, we look into second-order terms in the asymptotic expansion of the averaged RR iterates and obtain high probability bounds. We use these bounds to develop a new method that can accelerate the convergence rate of RR to $\bigO(\frac{1}{k^2})$ with high probability. Finally, we show that the $\bigO(\frac{1}{k^2})$ rate can also be achieved in expectation (which is a weaker notion of convergence with respect to convergence with high probability) for the $s=1$ case by adjusting the stepsize to the strong convexity constant of the objective properly.
\begin{figure}\hspace{-0.3in}
\vspace{-0.2in}
\begin{center}
	\includegraphics[width=1\linewidth]{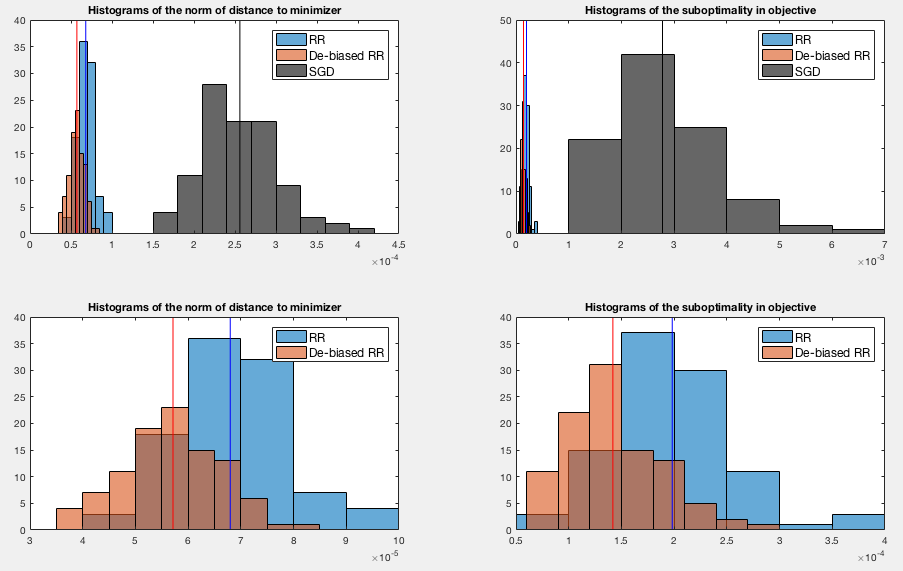}\label{sim-18}\vspace{-0.3in}
\caption{\label{fig-3}Comparison of RR, Debiased-RR (DRR) and SGD when component functions are random quadratics with $m=50$, $n=20$ and with simulation time 0.5 seconds over $500$ sample paths. Top, left: Histograms of $\dist_k$ for RR, DRR and SGD. Bottom, left: Histograms of $\dist_k$ for RR and DRR only (without SGD). Top, right: Histograms of the suboptimality in objective value for RR, DRR and SGD. Bottom, right: Histograms of the suboptimality in objective value for RR and DRR only (without SGD).}	
\vspace{0.3in}
	\includegraphics[width=1\linewidth]{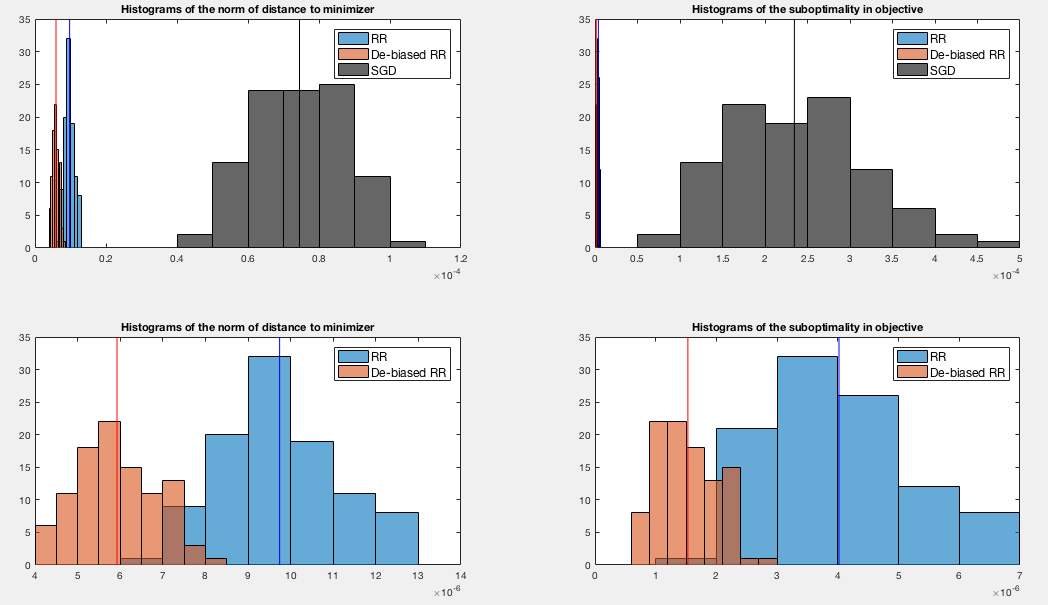}\label{sim-16}
\caption{\label{fig-4}Comparison of RR, De-biased-RR (DRR) and SGD. The simulation framework and parameters are the same as those in Fig. \ref{fig-3} except that the simulation time is 5 seconds instead for each path.}	
\end{center}
\end{figure}

\appendix 

\section{Proof of Theorem \ref{theo-expected-rate}}
\label{sec-appendix-exp-rate}
\begin{proof} Following the analysis of \cite{Gur2015IncGrad}, we could write
 	\beq x_0^{k+1} - x^* =  \big(I - \alpha_k P + \bigO(\alpha_k^3)\big) (x_0^k - x^*) - \alpha_k^2  \hat{\mu}_{\sigma_k} + \bigO(\alpha_k^3)
 	\label{iter-eval-quad-RR}
	\eeq 
where \beq \hat{\mu}_{\sigma_k} := -\underset{1\leq i < j\leq m}{\sum} P_{\sigma_k(j)} \nabla f_{\sigma_k(i)}(x_0^k). \label{def:M-sigma-k} 
\eeq
We also have \beqas
	\| \mu_{\sk} - \hat{\mu}_{\sk} \| &\leq& \underset{1\leq i < j\leq m}{\sum} \| P_{\sigma_k(j)} \| \| \nabla f_{\sigma_k(i)}(x_0^k) - \nabla f_{\sigma_k(i)}(x^*)\| 
								  \leq \underset{1\leq i < j\leq m}{\sum} L_{\sigma_k(j)} L_{\sigma_k(i)} \dist_k = \bigO(\dist_k)
\eeqas
where $\mu_{\sk}$ is defined by \eqref{def-Msigma} with $\sigma=\sigma_k$. Plugging this into \eqref{iter-eval-quad-RR}, 
	\beqas x_0^{k+1} - x^* ~=~ && \big(I - \alpha_k P + \bigO(\alpha_k^2) + \bigO(\alpha_k^3)\big) (x_0^k - 	x^*) - \alpha_k^2 \mu_{\sigma_k} 
		 + \bigO(\alpha_k^3 + \alpha_k^2 \dist_k).
 	\label{iter-eval-quad-RR-2}
	\eeqas 
 Taking norm squares of both sides in \eqref{iter-eval-quad-RR-2}, taking conditional expectations and using the fact that $\mu_{\sigma_k}$ is bounded (see \eqref{def-M-gamma-bound}), we obtain \beqa
	\E_{\sk} \left( \dist_{k+1}^2 \right | x_0^k) ~=~ &&(x_0^k - x^*)^T \big(I - 2\alpha_k P + \bigO(\alpha_k^2)\big) (x_0^k - x^*)  + 2\alpha_k^2 \langle x_0^k - x^*, -\bar{\mu}\rangle \nonumber \\
	 && + \bigO(\alpha_k^3\dist_k + \alpha_k^2 \dist_k^2 + \alpha_k^4) \label{exp-distance-bound}
     \eeqa
where $\E_{\sk}$ denotes the expectation with respect to the random permutation $\sigma_k$ and \[
	\bar{\mu} = \E_{\sk}\left(\mu_{\sigma_k}\right) = \E_{\sigma_1}\left(\mu_{\sigma_1}\right).
\]
It follows from Cauchy-Schwartz that for any $\beta>0$
   \beqas
    	\alpha_k^2 \left\| \langle x_0^k - x^*, -\bar{\mu}\rangle \right\| &\leq& \alpha_k^2 \dist_k \| \bar{\mu} \| 
    	    = \left(\sqrt{\beta} \alpha_k^{1/2}\dist_k\right) \frac{ \alpha_k^{3/2} \| \bar{\mu} \|}{\sqrt{\beta}} \leq \frac{\beta \alpha_k \dist_k^2}{2} + \frac{\alpha_k^3 \| \bar{\mu} \|^2 }{2\beta},
   \eeqas 
and also $$\alpha_k^3 \dist_k = \alpha_k^2 \left(\alpha_k \dist_k\right) \leq \frac{\alpha_k^4}{2} + \frac{\alpha_k^2 \dist_k^2}{2}.$$   
Plugging these bounds back into \eqref{exp-distance-bound}, using the lower bound \eqref{def-Hs} on the Hessian $H_* = P$ and invoking the tower property of the expectations:
	\beqas
	\E \left( \dist_{k+1}^2 \right) ~=~ &&\big(1 - \alpha_k (2c-\beta) + \bigO(\alpha_k^2)\big) \E \left(\dist_k^2\right)  + \alpha_k^3 \frac{\| \bar{\mu} \|^2}{\beta} + \bigO(\alpha_k^4). 
     \eeqas
Plugging in $\alpha_k = R/k^s$, it follows from Chung's lemma \cite[Lemma 4.2]{Fabian1967ChungLemma} that,
	\beq \E \left( \dist_{k+1}^2 \right) \leq \begin{cases} \frac{R^2 \|  \bar{\mu} \|^2}{\beta(2c-\beta)} \frac{1}{k^{2s}} + o\left(\frac{1}{k^{2s}}\right) & \mbox{if} \quad 0<s<1 \mbox{ and } 2c-\beta > 0,\\ \frac{R^3 \|  \bar{\mu} \|^2}{\beta (R\left(2c-\beta)-2\right)} \frac{1}{k^{2}}   + o(\frac{1}{k^{2}})  &  \mbox{if} \quad s=1 \mbox{ and } R(2c-\beta)-2 > 0, \end{cases} 
		\label{exp-dist-square-rate}
	\eeq
Next we choose $\beta$ to get the best upper bound above. This is done by choosing $\beta = c$ for $0<s<1$ and choosing $\beta = (Rc - 1)/R$ for $s=1$ which yields
 	\beq 
 		\E \left( \dist_{k+1}^2 \right) \leq \begin{cases} \frac{R^2 \|  \bar{\mu} \|^2}{c^2} \frac{1}{k^{2s}} + o\left(\frac{1}{k^{2s}}\right) & \mbox{if} \quad 0<s<1,\\ \frac{R^4 \|  \bar{\mu} \|^2}{(Rc-1)^2} \frac{1}{k^{2}}   + o(\frac{1}{k^{2}})  &  \mbox{if} \quad s=1 \mbox{ and } Rc -1 > 0. \end{cases} 
	\eeq	 
By Jensen's inequality, we have $\E(\dist_k) \leq \left( \E \left( \dist_{k+1}^2 \right) \right)^{1/2}$. Therefore, by taking square roots of both sides above in \eqref{exp-dist-square-rate} we conclude.
\end{proof}

\section{Technical lemmas for the proof of Theorem \ref{thm-averaged-rate-quad}} \label{sec:appendix-quad-case}

The first lemma is on characterizing what is the worst-case distance of the all the inner iterates of RR to the optimal solution $x^*$. This quantity we want to upper bound is a random variable, but the upper bounds we obtain are deterministic holding for every sample path. This lemma is based on Corollary \ref{coro-deter-decay-rate-dist} and uses the fact that the distance between the inner iterates are on the order of the stepsize.
 
\begin{lemma}\label{lem-random-tail-bound}Under the conditions of Theorem \ref{thm-averaged-rate-quad} we have 
    $\underset{0 \leq i < m}{\max} \| x_i^k - x^*\| =  \bigO(\frac{1}{k^s}).$
where $\bigO(\cdot)$ hides a constant that depends only on $\Gs, L, m,c$ and $R$. 
\end{lemma} 
\begin{proof} By Corollary \ref{coro-deter-decay-rate-dist}, 
   \beq \| x_0^k - x^* \|  = \bigO(\frac{1}{k^s}).
   		\label{upper-bound-cycle-begin}
   \eeq  
where $\bigO(\cdot)$ hides a constant that depends only on $\Gs, L, m,R$ and $c$.  We have also for any $0\leq i< m$ and $k\geq 0$,
   \beqas \| x_i^k - x^*\|  &\leq& \| x_0^k - x^*\| + \| x_i^k - x_0^k \|  = \| x_0^k - x^*\| + i\alpha_k \max_{\ell=1,\dots,i} \|\nabla f_{\sigma_k(\ell)} (x_{\ell-1}^k)\| \\
   &\leq& \| x_0^k - x^*\| + (m-1)\frac{R}{(k+1)^s} \big(  G_* + \max_{\ell=1,\dots,i} \|\nabla f_{\sigma_k(\ell)} (x_{\ell-1}^k) - \nabla f_{\sigma_k(\ell)} (x^*)\| \big) \\
   &\leq& \| x_0^k - x^*\| + (m-1)\frac{R}{(k+1)^s} \big(  G_* + L \max_{\ell=1,\dots,i} \|x_{\ell-1}^k - x^*\| \big). 
   \eeqas
where we used the $L$-Lipschitzness of the gradient of $f$ where $L$ is given by \ref{def-L}. Using \eqref{upper-bound-cycle-begin} and applying this inequality inductively for $i=0,1,2,\dots, m-1$ we conclude.
\end{proof}     

The second lemma is on characterizing how fast on average the outer iterates move (if normalized by the stepsize) after a cycle of the RR algorithm. This is clearly related to the magnitude of the gradients seen by the iterates and is fundamental for establishing the convergence rate of the averaged RR iterates in Theorem \ref{thm-averaged-rate-quad}. 

\begin{lemma}\label{lem-I-lk} Under the conditions of Theorem \ref{thm-averaged-rate-quad}, consider the sequence
\beq I_{q,k} = \frac{\sum_{j=(1-q)k}^{k-1} {(x_0^j - x_{0}^{j+1})}{\alpha_j^{-1}}}{qk}, \quad  0 < q \leq 1. \label{def-I-qk}
\eeq
Then, 
    $ I_{q, k} = \begin{cases}  \bigO \big( \frac{\log k}{k}  \big) 
     &  \mbox{if} \quad q=1, \\  
    			\bigO \big( \frac{1}{k}\big) 	 & \mbox{if} \quad 0<q<1.  
    \end{cases}  
   $
In the former case, $\bigO(\cdot)$ hides a constant that depends only on $\Gs, L,m,  c, R, s,q$ and $\dist_0$. In the latter case, the same dependency on the constants occurs except that the dependency on $\dist_0$ can be removed. 
\end{lemma}
\begin{proof}  It follows from integration by parts that for any $\ell<k$, 
\beq -\sum_{j=\ell}^{k-1}{(x_0^j - x_{0}^{j+1})}{\alpha_j^{-1}}  =  \alpha_k^{-1} (x_0^k - x^*) - \alpha_\ell^{-1} (x_0^\ell - x^*) - \sum_{j=\ell}^{k-1} (x_0^{j+1} - x^*)(\alpha_{j+1}^{-1}-\alpha_j^{-1}). \label{sum-of-aver-iters}
\eeq
Next, we investigate the asymptotic behavior of the terms on the right-hand side. A consequence of Corollary \ref{coro-deter-decay-rate-dist} and the inequality \ref{def-M-gamma} is that  
	\beq \alpha_k^{-1} \| x_0^k -x^*\| = \frac{(k+1)^s}{R} \| x_0^k -x^*\| \leq \frac{LmG_*}{c} + o(1) = \bigO(1)
		\label{step-times-dist-bd} 	
	\eeq
and therefore
\beqas 
	| \alpha_{k+1}^{-1} - \alpha_{k}^{-1} |  \| x_0^k -x^*\| &=& \frac{(k+2)^s - (k+1)^s} {(k+1)^s} \alpha_k^{-1}\| x_0^k -x^*\|  
	 = \bigg( \big(1 + \frac{1}{k+1} \big)^s -1 \bigg) \alpha_k^{-1}\| x_0^k -x^*\| \\
	  &\leq &  \frac{s}{k+1}\alpha_k^{-1}\| x_0^k -x^*\|  
	  \leq \frac{sLmG_*}{c} \frac{1}{k+1} + o(\frac{1}{k+1}) = \bigO(\frac{1}{k+1})
\eeqas
where $\bigO(\cdot)$ hides a constant that depends only on $L,\Gs, c, m$ and $s$. Then, setting  $\ell = (1-q)k$ in \eqref{sum-of-aver-iters}, it follows that 
\beqa  \big\| \sum_{j=\ell}^{k-1} {(x_0^j - x_0^{j+1})}{\alpha_j^{-1}} \big\| &\leq&  \|\alpha_k^{-1} (x_0^k - x^*)\| + \| \alpha_{(1-q)k}^{-1} (x_0^{(1-q)k}- x^*)\| \\
  && + \sum_{j=(1-q)k}^{k-1} \| x_0^{j+1} - x^*\| |\alpha_{j+1}^{-1}-\alpha_j^{-1}|.  \nonumber\\
       &=& \bigO(1) + \| \alpha_{(1-q)k}^{-1} (x_0^{(1-q)k} - x^*)\| + \bigO \bigg( \sum_{j=(1-q)k}^{k-1} \frac{1}{j+1}\bigg).  \label{weighted-iter-diff-bound} 
\eeqa
We also have
  \beq  \| \alpha_{(1-q)k}^{-1} (x_0^{(1-q)k} - x^*)\|  = \begin{cases} \alpha_0^{-1} \dist_0 & \mbox{if} \quad q = 1, \\ \bigO(1)  & \mbox{if} \quad 0<q<1, \end{cases} \label{init-averaged-iter-bound}
  \eeq
where the second part follows from \eqref{step-times-dist-bd} with similar constants for the $\bigO(\cdot)$ term. As the sequence $\frac{1}{j+1}$ is monotonically decreasing, for any $k>0$ we have the bounds 
\beq  \sum_{j=(1-q)k}^{k-1} \frac{1}{j+1} \leq \frac{1}{(1-q)k + 1} + \int_{(1-q)k}^{k-1} \frac{1}{x+1} dx \leq \begin{cases} 1 + \log k & \mbox{if } \quad  q=1, \\
 1 + \log(\frac{1}{1-q})  & \mbox{if } \quad  0 <q <1.
\end{cases}  \label{tail-suffix-averaging}
\eeq
Note that when $q=1$ this bound grows with $k$ logarithmically whereas for $q<1$ it does not grow with $k$. 
Then, combining \eqref{weighted-iter-diff-bound}, \eqref{init-averaged-iter-bound} and \eqref{tail-suffix-averaging}  we obtain
  $$ \| \Iqk \| \leq \frac{\big\| \sum_{j=\ell}^{k-1} {(x_0^j - x_0^{j+1})}{\alpha_j^{-1}} \big\|}{qk} = \begin{cases}  \bigO \big(  \frac{\log k}{k}\big) 
     &  \mbox{if} \quad q=1 \\  
    			\bigO \big( \frac{1}{k}\big) 	 & \mbox{if} \quad 0<q<1  
    \end{cases}  $$
as desired which completes the proof.    
\end{proof}

\begin{lemma}\label{lem-computing-mu-bar} Let $\sigma$ be a random permutation of $\{1,2,\dots,m\}$ sampled uniformly over the set of all permutations $\Gamma$ defined by \eqref{def-Gamma-perm-set} and $\mu(\sigma)$ be the vector defined by \eqref{def-Msigma} that depends on $\sigma$. Then, \beq
	\bar{\mu} = \E_{\sigma}\big(\mu({\sigma})\big) = \frac{1}{2} \sum_{i=1}^m P_i \nabla f_i(x^*)
\eeq
where $\E_{\sigma}$ denotes the expectation with respect to the random permutation $\sigma$ and $\bar{mu}$ is defined by \eqref{def-mu-bar}.
\end{lemma}
\begin{proof} For any $i \neq \ell$, the joint distribution of $(\sigma(i), \sigma(\ell))$ is uniform over the set of all (ordered) pairs from $\{1,2,\dots,m\}$. Therefore, for any $i \neq \ell$, 
    \beqas
    \E_\sigma \big[ P_{\sigma(i)} \nabla f_{\sigma(\ell) }(x^*) \big] &=& \sum_{i=1}^m \sum_{i\neq j, j=1}^m \frac{P_i\nabla f_j(x^*)}{m(m-1)} \\
    &=& \frac{\sum_{i=1}^m P_i  \sum_{j=1}^m \nabla f_j(x^*) - \sum_{j=1}^m P_j \nabla f_j(x^*)}{m(m-1)} 
    =  - \frac{\sum_{j=1}^m P_j \nabla f_j(x^*)}{m(m-1)}
    \eeqas 
where we used the fact that $\nabla f(x^*) = \sum_{j=1}^m \nabla f_j(x^*)=0$ by the first order optimality condition. Then, by taking the expectation of \eqref{def-v-sigma-k}, we obtain 
\beqas \E_\sigma (\mu(\sigma)) &=& - \sum_{i=1}^{m} \sum_{\ell=0}^{i-1} \E \big[ P_{\sigma(i)} \nabla f_{\sigma(\ell)}(x^*) \big] 
        =  \sum_{i=1}^{m} \sum_{\ell=0}^{i-1}  \frac{\sum_{j=1}^m P_j \nabla f_j(x^*)}{m(m-1)} = \frac{\sum_{j=1}^m P_j \nabla f_j(x^*)}{2}.
\eeqas	 
which completes the proof.
\end{proof}
\begin{lemma} \label{lem-martingale-variance} Under the conditions of Theorem \ref{thm-averaged-rate-quad},  the following statements are true: 
\begin{itemize} 
   \item [$(i)$] We have  \beq E_k = \alpha_k \mu(\sigma_k)  + \bigO(\alpha_k^2), \quad k\geq 0,  \label{Ek-expression} \eeq 
where $E_k$ is the gradient error defined by \eqref{def-Ek}, $\bigO(\cdot)$ hides a constant that depends only on $\Gs, L,m, R$ and $c$ and
   \beq  \mu(\sk) = - \sum_{i=1}^{m} P_{\ski} \sum_{\ell=1}^{i-1} \nabla f_{\skl}(x^*). 
		\label{def-v-sigma-k}   
   \eeq
is a sequence of i.i.d. variables where the function $\mu(\cdot)$ is defined by \eqref{def-Msigma}. 
   \item [$(ii)$] For any $0 <q\leq 1 $,  $\lim_{k \to \infty} Y_{q,k}=  \bar{\mu}$ a.s. where 
    $ Y_{q,k} = \frac{\sum_{i=(1-q)k}^{k-1} E_j}{\sum_{j=(1-q)k}^{k-1} \alpha_j}.$
   \item [$(iii)$] It holds that 
       \beq  \| \mu(\sigma_k) \| \leq Lm G_*.	
       		\label{def-mu-sk-upper-bound}
       \eeq	
\end{itemize}

\end{lemma}
\begin{proof} 

	\begin{itemize} 
			\item [$(i)$] As component functions are quadratics, \eqref{def-Ek} becomes
\beqas E_k &=& \sum_{i=1}^{m} P_{\ski}(x_{i-1}^k - x_0^k) = - \sum_{i=1}^{m}  P_{\ski}  \alpha_k \sum_{\ell=1}^{i-1} \nabla f_{\skl } (x_{\ell-1}^k). 
 \eeqas
where we can substitute 
 \beq  \nabla f_{ \skl } (x_{\ell-1}^k) =  \nabla f_{ \skl } (x^*)  + P_{\skl} (x_{\ell-1}^k - x^*).
 	\label{quadratic-only-one}
 \eeq
Then an application of Lemma \ref{lem-random-tail-bound}  proves directly the desired result.
   			\item [$(ii)$] We introduce the normalized gradient error sequence $Y_j = E_j / \alpha_j$. By part $(i)$, $Y_j = \mu(\sigma_j) + \bigO(\alpha_j)$ where $\mu(\sigma_j)$ is a sequence of i.i.d. variables. By the strong law of large numbers, we have
   			     \beq \lim_{k \to \infty} \frac{\sum_{j=0}^{k-1} \mu(\sigma_j)}{k} = \E \mu(\sigma_j) = \bar{\mu} \quad \mbox{a.s.}\label{lim-aver-vuj}  \eeq
where the last equality is by the definition of $\bar{\mu}$. Therefore,   			
    			 \beqas \lim_{k \to \infty} \frac{\sum_{j=0}^{k-1} Y_j}{k} = \lim_{k \to \infty} \bigg(\frac{\sum_{j=0}^{k-1} \mu(\sigma_j)}{k} + \frac{\sum_{j=0}^{k-1} \bigO(\alpha_j)}{k}\bigg) = \bar{\mu} \quad \mbox{a.s.} \eeqas
where we used the fact that the second term is negligible as $\sum_{j=0}^{k-1} \alpha_j / k = \bigO(k^{-s}) \to 0$. 
As the average of the sequence $Y_j$ converges almost surely, one can show that this implies almost sure convergence of a weighted average of the sequence $Y_j$ as well as long as weights satisfy certain conditions as $k \to \infty$. In particular, as the sequence $\{\alpha_j\}$ is monotonically decreasing and is non-summable, by \cite[Theorem 1]{Etemadi:2006ko}, 
       \beq \lim_{k \to \infty} Y_{1,k} = \lim_{k \to \infty}\frac{\sum_{j=0}^{k-1} \alpha_j Y_j}{\sum_{j=0}^{k-1} \alpha_j} =  \lim_{k \to \infty} \frac{\sum_{j=0}^{k-1} E_j}{\sum_{j=0}^{k-1} \alpha_j} = \bar{\mu} \quad \mbox{a.s.}\label{lim-average-mu-star}
       \eeq 
This completes the proof for $q=1$. For $0<q<1$, by the definition of $Y_{q,k}$, we can write $Y_{1,k} = (1-w_k) Y_{q,k} + w_k Y_{1, (1-q)k}$ where the non-negative weights $w_k$ satisfy    
$$w_k = \frac{\sum_{j=0}^{(1-q)k - 1 } \alpha_j}{\sum_{j=0}^{k - 1 } \alpha_j } \to_{k \to \infty} (1-q)^{1-s} < 1.$$     
    As both $Y_{1,k}$ and $Y_{1, (1-q)k}$ go to $\bar{\mu}$ a.s. by \eqref{lim-average-mu-star}, it follows that 
$$\lim_{k\to \infty} Y_{q,k} = \lim_{k\to \infty}  \frac{Y_{1,k} - w_k Y_{1,(1-q)k}}{1-w_k} = \bar{\mu}\quad \mbox{a.s}$$ as well for any $0<q<1$. This completes the proof.
	\item [$(iii)$] This is a direct consequence of the triangle inequality applied to the definition \eqref{def-v-sigma-k} with $L_i = \|P_i\|$ and $L=\sum_{i=1}^m L_i$. 
\end{itemize}
\end{proof}

\section{Techical Lemmas for the proof of Theorem \ref{thm-averaged-random-rate}}\label{sec:appendix-general-case}
We first state the following result from \cite{Gur2015IncGrad} which extends Corollary \ref{coro-deter-decay-rate-dist} from quadratics to smooth functions. 

\begin{corollary}\label{coro-deter-decay-rate-dist-smooth} Under the setting of Theorem \ref{thm-averaged-random-rate}, then
		\beqs \  \dist_k \leq \frac{R M}{c}\frac{1}{k^s} + o(\frac{1}{k^s}), 
         \eeqs 
where the right-hand side is a deterministic sequence, $M := Lm \Gs$ and $\Gs$ is defined by \eqref{def-Gs}.
\end{corollary}
\begin{proof} The proof of Corollary \ref{coro-deter-decay-rate-dist} is based on Theorem \ref{theo-rate-quadratics} from \cite{Gur2015IncGrad}. This theorem admit an extension to smooth functions with Lipschitz gradients to (see \cite{Gur2015IncGrad}), therefore by the same reasoning along the lines of Corollary \ref{coro-deter-decay-rate-dist} the result follows.
\end{proof}

\begin{lemma}\label{lem-random-tail-bound-2} Under the conditions of Theorem \ref{thm-averaged-random-rate}, all the conclusions of Lemma \ref{lem-random-tail-bound} remain valid. 
\end{lemma} 
\begin{proof} The proof of Lemma \ref{lem-random-tail-bound}  
applies identically except that instead of Corollary \ref{coro-deter-decay-rate-dist} we use its extension Corollary \ref{coro-deter-decay-rate-dist-smooth}.
\end{proof}     

\begin{lemma}\label{lem-I-Ik-2} Under the conditions of Theorem \ref{thm-averaged-random-rate}, all the conclusions of Lemma \ref{lem-I-lk} remain valid. 
\end{lemma} 
\begin{proof} The proof of Lemma \ref{lem-I-lk} applies identically with the only difference that the bound on $\dist_k = \|x_0^k - x^*\|$ is obtained from Corollary \ref{coro-deter-decay-rate-dist-smooth} instead of Corollary \ref{coro-deter-decay-rate-dist}.
\end{proof}

\begin{lemma} \label{lem-martingale-variance-2} Under the conditions of Theorem \ref{thm-averaged-random-rate},  the following statements are true: 
\begin{itemize} 
   \item [$(i)$] We have  \beq E_k = \alpha_k \vb(\sigma_k)  + \bigO(\alpha_k^2), \quad k\geq 0,  \label{Ek-expression-2} \eeq 
where  $\bigO(\cdot)$ hides a constant that depends only on $\Gs, L,m, R,  c$ and $\hlip$ and
   $$ \vb (\sk) = - \sum_{i=0}^{m-1} \nabla^2 f_{\ski} (x^*) \sum_{\ell=0}^{i-1} \nabla f_{\skl}(x^*). $$ 
   \item [$(ii)$]  It holds that   
   \beq
          \|\vb(\s_k)\| \leq LmG_* 
          		\label{v-inf-norm-2}
    \eeq
where    
    \beq      
    \mub := \E \vb (\sk) = {\sum_{i=1}^m \nabla^2 f_i(x^*) \nabla f_i(x^*)}/{2}. \label{def-mu-star-2} 
   \eeq
   \item [$(iii)$] For any $0 <q\leq 1 $,  $\lim_{k \to \infty} Y_{q,k}=  \mub$ with probability one where 
    \beq Y_{q,k} = \frac{\sum_{i=(1-q)k}^{k-1} E_j}{\sum_{j=(1-q)k}^{k-1} \alpha_j}.\	
		\label{def-Yqk-2}    
     \eeq
\end{itemize}  
\end{lemma}

\begin{proof} For part $(i)$, first we express $E_k$ using the Taylor expansion and the Hessian Lipschitzness as 
\beqas E_k &=& \sum_{i=1}^{m} \bigg( \nabla^2 f_{\ski}(x_0^k) \bigg) (x_{i-1}^k - x_0^k) + \bigO( \hlip \|x_{i-1}^k - x_0^k \|^2 ). \\
            &=& - \sum_{i=1}^{m} \bigg( \nabla^2 f_{\ski}(x_0^k) \bigg) (x_{i-1}^k - x_0^k) + \bigO \bigg( \alpha_k^2 \hlip \bigg\|  \sum_{\ell=1}^{i-1} \nabla f_{ \skl } (x_{\ell-1}^k)  \bigg\| \bigg)
 \eeqas
 By Lemma \ref{lem-random-tail-bound-2},  we have $\|x_\ell^k - x^*\| = \bigO(\alpha^k)$ with probability one. Then, by the gradient and Hessian Lipschitzness we can substitute above
\beqas \nabla f_{ \skl } (x_{\ell-1}^k) =  \nabla f_{ \skl } (x^*)  + \bigO(\alpha^k), \quad \nabla^2 f_{ \skl } (x_{\ell-1}^k) =  \nabla^2 f_{ \skl } (x^*)  + \bigO(\alpha^k). 
\eeqas
which implies directly Equation \eqref{Ek-expression-2}. The rest of the proof for parts $(ii)$ and $(iii)$ is similar to the proof of Lemma \ref{lem-martingale-variance} and is omitted.
\end{proof}

\bibliographystyle{plain}
\bibliography{rand_shuffling}
\end{document}